\documentclass[twoside,11pt]{article}

\usepackage[abbrvbib,preprint]{jmlr2e}

\usepackage{calligra}
\usepackage{cancel}
\usepackage{graphicx,psfrag,epsf,bbm}
\usepackage{mathtools}
\usepackage{enumerate}
\usepackage{url} 
\usepackage{microtype}

\usepackage{calligra}
\usepackage[linesnumbered,ruled,vlined]{algorithm2e}
\usepackage[noend]{algpseudocode}
\usepackage[nothm]{YK}
\RequirePackage{amsmath,bbm}

\usepackage{algpseudocode}
\usepackage{cancel}
\usepackage{graphicx}
\usepackage{mathtools}
\usepackage[ruled,vlined]{algorithm2e}
\SetKwRepeat{Do}{do}{while}

\usepackage{enumitem}

\ShortHeadings{Minimax optimal approaches to label shift}{Maity, Sun and Banerjee}
\firstpageno{1}

\newcommand{\datalabeled}{\cD_{\text{L}}}
\newcommand{\dataunlabeled}{\cD_{\text{U}}}

\newcommand{\samplespacelabeled}{\cS_{\text{L}}}
\newcommand{\samplespaceunlabeled}{\cS_{\text{U}}}

\newcommand{\indicator}{\mathbbm{1}}
\newcommand{\TN}{\mathrm{TN}}
\newcommand{\betafloor}{\lfloor \beta\rfloor}
\newcommand{\upn}{^{(n)}}

\usepackage{lastpage}
\ShortHeadings{Minimax optimal approaches to the label shift problem}{Maity, Sun, and Banerjee}

\begin{document}

\title{Minimax optimal approaches to the label shift problem in non-parametric settings}
\author{\name Subha Maity \email smaity@umich.edu \\
        \name Yuekai Sun \email yuekai@umich.edu\\
        \name Moulinath Banerjee \email moulib@umich.edu\\
       \addr Department of Statistics\\
       University of Michigan \\
       Ann Arbor, MI}
        \editor{Amos Storkey}
    \date{}
   
  \maketitle

\begin{abstract}
We study the minimax rates of the label shift problem in non-parametric classification. In addition to the unsupervised setting in which the learner only has access to unlabeled examples from the target domain, we also consider the setting in which a small number of labeled examples from the target domain is available to the learner. Our study reveals a difference in the difficulty of the label shift problem in the two settings, and we attribute this difference to the availability of data from the target domain to estimate the class conditional distributions in the latter setting. We also show that a class proportion estimation approach is minimax rate-optimal in the unsupervised setting.
\end{abstract}


\section{Introduction}
A key feature of general intelligence is the transfer of knowledge from one task to another similar but non-identical task. However, most machine learning (ML) methods are not designed for such out-of-distribution (OOD) generalization. This limitation has led to  embarrassing performances of ML models in several high-profile applications \citep{angwin2016Machine,dastin2018Amazon}. Transfer learning attempts to improve the OOD generalization performance of ML models and has attracted attention in a variety of application areas: computer vision \citep{tzeng2017adversarial,gong2012geodesic}, speech recognition \citep{huang2013cross} and genre classification \citep{choi2017transfer}. Transfer learning is also known as domain adaptation, and we refer to \citet{pan2009survey,weiss2016survey} for surveys of transfer learning.

Despite its empirical success, there is limited knowledge of the fundamental limits of transfer learning. In this paper, we study the fundamental limits of transfer learning for label shift problem under a binary classification setup. We posit the learner has a (labeled) training dataset from the source distribution/domain $P$ and either a small labeled dataset or an unlabeled dataset from the target domain $Q$. The learner knows $P$ is similar, but not identical to $Q$ (the differences between $P$ and $Q$ will be made precise later). The learner's task is to combine this knowledge about the similarities between $P$ and $Q$ with data from $P$ for inference in $Q$. 


At a high-level, there are two lines of theoretical work in transfer learning. The first line of work 
focuses on obtaining bounds on the worst-case performance of ML models in similar target domains \citep{ben-david2010theory,mansour2009Domain}. Here similarity between the source and target domain is measured by some notion of distance/divergence between probability distributions. Although general, such bounds are usually pessimistic, especially if the learner has access to some data from the target domain.


The second line of work focuses on problems in which the learner has some data from the target domain. In order to leverage the data from the source domain, the learner must make some assumptions on the similarities between the source and target domains. Such assumptions usually take the form of invariances between the source and target domain. To keep things simple, we assume it is possible to partition each sample $Z$ from the source and target domains into two parts: $Z = (U,V)$. We factorize the source (resp.\ target) distributions as $P(U,V) = P(U\mid V)P(V)$ (resp.\ $Q(U,V) = Q(U\mid V)Q(V)$). By picking $U$ and $V$ carefully, it is possible to obtain many common transfer learning settings.

If the conditional factor changes between the source and target domains ($P(U\mid V) \ne Q(U\mid V)$) while the marginal factor remains the same ($P(V) = Q(V)$), then there is \emph{conditional drift} between the source and the target domains. A prominent example of conditional drift is posterior drift \citep{cai2019Transfer,maity2021linear}, in which the marginal distribution of the features $X$ remains the same between the source and target domains, but the conditional distribution of the label / response given the features changes. On the other hand, if the marginal factor changes between the source and the target domains ($P(V) \ne Q(V)$) while the conditional factor remains the same ($P(U\mid V) = Q(U\mid V)$), then there is \emph{marginal drift} between the source and target domains. Prominent examples of marginal drift include covariate shift \citep{kpotufe2018Marginal, zhang2015multi} ($U = Y$, $V = X$) and label shift \citep{storkey2009training,saerens2002adjusting,lipton2018Detecting,scholkopf2012causal, zhang2015multi} ($U = X$, $V = Y$).

In this paper, we focus on the label shift problem: $P(X\mid Y) = Q(X\mid Y)$ but $P(Y) \ne Q(Y)$. This problem arises in many application areas. For example, consider building a pneumonia detector. While the symptoms of pneumonia may not change from month to month, the prevalence of the disease in the population may increase in a month of pandemic (\eg\ January). There are two version of the label shift problem: the supervised version in which the learner has labeled data from the target domain and the unsupervised version in which the learner only has unlabeled data from the target domain. There has been a flurry of recent work (\eg\ \citet{lipton2018Detecting,azizzadenesheli2019Regularized,garg2020Unified}) on methods for the label shift problem, especially the unsupervised version. Our work complements this line of work by studying the fundamental limits of the label shift problem in a non-parametric setting.

The organization of rest of the paper follows: we formulate the supervised and unsupervised label shift problems and state the working assumption in Section \ref{sec:set-up}. We study the fundamental limits of the supervised and unsupervised label shift problems in Sections \ref{sec:labeled-target} and \ref{sec:unlabeled-target-data} respectively. In Sections \ref{sec:simulations} and \ref{sec:lb-proof}, we present simulation studies that confirm our theoretical results and prove the lower bound for the unsupervised label shift problem. Finally, we wrap up with a brief discussion of the implications of our results in section \ref{sec:discussion}.


\section{Setup}
\label{sec:set-up}
In this section, we forumulate the label shift problem and state the working assumptions. 

\subsection{Notations and definitions} For a random vector $ (X,Y)\in [0,1]^d \times \{0,1\} $ with distribution $ P $, we denote the marginal distribution of $ X $ by $ P_X $ and the marginal probability of the event $ \{Y=1\} $ by $ \pi_P. $ We denote the support of $P$ with $ \text{supp}(P) $. We use $ \indicator $ to denote the indicator function taking the value in $ \{0,1\}. $ We use the $\wedge\vee$ notation for min and max: $ a\wedge b \triangleq \min (a,b) $ and $ a\vee b \triangleq \max(a,b). $ Finally, $ \lambda(\cdot) $ denotes the Lebesgue measure of a set in a Euclidean space, and $ B(x,r) $ denotes the $ d $-dimensional (closed) Euclidean ball of radius $ r>0 $ with center $ x \in \reals ^d $. For a generic probability distribution $\mu$ the notation $X_1, \dots, X_n\overset{\ind}{\sim}\mu$ implies that $X_1, \dots X_n$ are independently distributed and each of them have distribution $\mu$.

\subsection{Label shift in nonparametric classification}
\label{subsec:set-up}
Let $ P $ and $ Q $ be two distributions on $ [0,1]^d \times \{0,1\}. $ We consider $P$ as the distribution of the samples from the source domain and $Q$ as that of the samples from the target domain. In the supervised version of the label shift problem, the learner has labeled data from the source and target domains:
\begin{align*}
  \datalabeled \triangleq & \Big\{ (X_1^P, Y_1^P), \dots (X_{n_P}^P, Y_{n_P}^P) \overset{\ind}{\sim}  P; \\ &  
   (X_1^Q, Y_1^Q), \dots (X_{n_Q}^Q, Y_{n_Q}^Q) \overset{\ind}{\sim}  Q    \Big\} \in \left( \cX\times \cY \right)^{(n_P+n_Q)} . 
\end{align*}
On the other hand, in the unsupervised version of the problem, the learner only has unlabeled data from the target domain:
\begin{align*}
  \dataunlabeled \triangleq &  \left\{ (X_1^P, Y_1^P), \dots (X_{n_P}^P, Y_{n_P}^P) \overset{\ind}{\sim}  P; \ X_1^Q, \dots X_{n_Q}^Q \overset{\ind}{\sim}  Q_X    \right\} \\
  & \in \left( \cX\times \cY \right)^{ n_P} \times \cX^{ n_Q}.  
\end{align*}
In both versions of the problem, the class conditionals in the source and target domains are identical: $ P(\cdot | Y) = Q(\cdot | Y) $. However, the (marginal) distributions of the labels differ: $ \pi_P\neq \pi_Q. $ 

Let $ G_0 \triangleq P(\cdot\mid Y=0)$ and $ G_1 \triangleq P(\cdot\mid Y=1)$ be the class conditionals. Note that in light of the equivalence of class conditionals in the source and target domains, we can (equivalently) define $ G_0 \triangleq Q(\cdot\mid Y=0)$ and $ G_1 \triangleq Q(\cdot\mid Y=1)$. The regression functions in the source and target domains are 
\[ \eta_P(x) \triangleq \begin{cases}
P(Y = 1|X=x) & \text{if } x\in \text{supp}(P_X)  \\
\frac12 & \text{otherwise}
\end{cases} \]  \[ \eta_Q(x) \triangleq \begin{cases}
Q(Y = 1|X=x) & \text{if } x\in \text{supp}(Q_X)  \\
\frac12 & \text{otherwise}
\end{cases}. \] 

In both versions of the label shift problem, the goal of the learner is to correctly classify samples from the target domain: learner wishes to learn a classifier $ \hat f :[0,1]^d \to \{0,1\} $ from the available data ($\datalabeled$ in the supervised version and $\dataunlabeled$ in the unsupervised version) that minimizes the classification error rate in the target domain $ Q (Y\neq \hat f(X))$. The Bayes classifier in the target domain is  
\[ 
f_Q^*(x) = \begin{cases}
0 & \text{if } \eta_Q(x) \le \frac12,\\
1 & \text{otherwise;}
\end{cases} 
\] 
\ie\ $f^*_Q \in \argmin_{h \in \cH } \Pr_Q (Y \neq h(X))$, where $\cH$ is the set of all measurable functions $ h: [0,1]^d \to \{0,1\} $. We consider the error rate of the Bayes classifier as a baseline and study the \textbf{excess risk} of learned classifier $\hat f$: 
\[  
\cE_Q (\hat f) = Q(Y \neq \hat f(X)) - Q(Y \neq f^*_Q(X)). 
\]  
The excess risk is a random quantity depending on the available data through the classifier $ \hat f$ and is known to have the following representation \citep{gyorfi1978rate}: 
\begin{equation}
\cE_Q(\hat f) = 2 \Ex_Q \left[\left| \eta_Q(X) - \frac12 \right| \indicator \{ \hat f(X) \neq f^*_Q(X) \}\right].
\end{equation}

To keep things simple, we assume that $G_0$ and $G_1$ are absolutely continuous with respect to the Lebesgue measure on $\reals^d$. We also assume that the marginal distribution of the features in the target domain $Q_X = \pi_QG_1 + (1-\pi_Q)G_0$ satisfies the \emph{strong density condition}. This is a common assumption in non-parametric classification (see \cite[Definition 2.2]{audibert2007fast}). 

\begin{definition}[strong density condition]
\label{def:strong-density}
A distribution $P$ defined on $\reals^d$ satisfies the strong density condition with parameters $ \mu_-, \mu_+, c_\mu, r_\mu>0 $ if
\begin{enumerate}
\item $P$ is absolutely continuous with respect to the Lebesgue measure on $ \reals^d$;
\item its support is regular: $\text{supp}(P)$ is compact and  $\lambda \left[ \text{supp}(P)  \cap B(x,r)  \right] \ge c_\mu \lambda[B(x,r)]$ for all $0 <r\le r_\mu$ and $x \in \text{supp}(P)$;
\item $\mu_- < \frac{d P}{d\lambda}(x) < \mu_+$ for all $x\in \text{supp}(P).$
\end{enumerate}
\end{definition}


We note that we only impose the strong density condition in the target domain because that is the domain in which we wish to study the performance of the classifier. Let $g_0$ and $g_1$ be the densities of $ G_0 $ and $ G_1 $ respectively (with respect to the Lesbegue measure on $\reals^d$). In terms of $g_0$ and $g_1$, the regression function in the target domain is
\begin{equation}
\eta_Q(x) = \begin{cases}
\frac{\pi_Q g_1 (x)}{\pi_Q g_1(x)+(1-\pi_Q)g_0(x) } & \text{if }  x \in \text{supp}(Q_X)  \\
\frac12 & \text{otherwise.}
\end{cases} 
\label{eq:regression-function}
\end{equation}


Inspecting \eqref{eq:regression-function}, we see that the main difficulty in estimating $\eta_Q$ is estimating the class conditional densities $ g_0$ and $g_1 $. The symmetry in the problem suggests that errors in estimating $g_0$ and $g_1$ affect the convergence rate of the excess risk equally. If the classes are imbalanced, then the excess risk depends on the estimation error of the rarer class. To measure class imbalance in our problem setup, we assume $\pi_P\in[\eps_P,1-\eps_P]$ for some $\eps_P > 0$. As we shall see, $\eps_P$ affects the effective sample size from the source domain in the minimax rate.

To quantify the hardness of estimating the class conditional densities $g_0$ and $g_1$, we impose standard smoothness conditions on them. For any $ s = (s_1, \dots , s_d)\in \bbN^d $ and $ x = (x_1, \dots , x_d) \in \text{supp}(Q_X) ,$ define $ |s| \triangleq s_1+\dots + s_d, \ s! \triangleq s_1!\dots s_d! $ and $ x^s = x_1^{s_1}\dots x_d^{s_d}. $ Let $ D^s $ denote the differential operator \[ D^s = \frac{\partial^{s_1+\dots + s_d}}{\partial x_1^{s_1}\dots \partial x_d^{s_d}}. \] For any $ g:\text{supp}(Q_X)\to\reals$ that is $ \lfloor \beta \rfloor $-times continuously differentiable at a point $ x_0\in \text{supp}(Q_X), $ we denote by $ g_{x_0}^{(\beta)} $ its Taylor expansion of degree $ \lfloor \beta \rfloor  $ at $ x_0 $:
\[ g_{x_0}^{(\beta)}(x) = \sum_{|s|\le \lfloor \beta \rfloor} \frac{(x-x_0)^s}{s!} D^s g(x_0). \]

\begin{definition}
[H\"older class]
\label{def:holder-smooth}
A function $ g:\Omega \to \reals $ is called $ (\beta , L, r_0) $-H\"older smooth ($\beta$-H\"older smooth in short)  if it is  $ \betafloor $-times continuously differentiable and satisfies \[ |g(x) - g_y^{(\beta)}(x)| \le L\|x-y\|_2^\beta \hspace{0.5in} \forall x\in B(y, r_0) \cap \Omega , \ y \in \Omega.  \] We denote the set of all such functions as   $ \Sigma(\beta, L, r_0) .$ 
\end{definition}




As we saw, the difficulty of the label shift problem depends on the hardness of estimating the class conditionals $g_0$ and $g_1$, so we assume they are $\beta$-H\"{o}lder smooth. Prior studies on the fundamental limits of covariate shift \citep{kpotufe2018Marginal} and posterior drift \citep{cai2019Transfer} have imposed similar smoothness conditions on the regression function instead of the class conditionals. This is because in those problems the Bayes classifier remains the same in the source and target domains:
\[\textstyle
\{x\in[0,1]^d\mid\eta_P(x) \le \frac12\} = \{x\in[0,1]^d\mid\eta_Q(x) \le \frac12\}.
\]
As the hardness of the transfer learning problem now depends on the hardness of estimating $\eta_P$,  it is important in both covariate shift and posterior drift to quantify the hardness of estimating $\eta_P$.  

\vspace{5mm}

We further introduce the margin condition to quantify the difficulty of the classification task in the target domain. It was introduced in \citet{tsybakov2004optimal} and adapted by \citet{audibert2007fast, cai2019Transfer, kpotufe2018Marginal} to study the convergence rate of the excess risk in binary classification problems. Intuitively, the condition restricts the probability mass around the Bayes decision boundary (region of the feature space such that $\eta_Q(x) \approx \frac12$). In other words, it implies $ \eta_Q(X)$ is far from $ \frac12 $ with appreciably high probability.

\begin{definition}
[margin condition for $ Q $]
\label{def:margin-condition}
The distribution $ Q $ satisfies the margin condition with parameter $ \alpha, $ if there exist $c_\alpha,  C_\alpha>0, $ such that  \[\text{for all } 0<t<c_\alpha, \hspace{0.3cm}  Q_X \left( 0< \left | \eta_Q(X) - \frac12 \right | \le t  \right) \le C_\alpha t^\alpha. \]
\end{definition}


We note that the condition becomes more stringent as $\alpha$ grows. We also note that if $ Q_X $ satisfies the strong density assumption and $ \beta (1\wedge\alpha) > 1, $ then there is no distribution $ Q $ such that the regression function $ \eta_Q $ crosses $ \frac12 $ in the interior of the support of $ Q_X $ (see \cite[Proposition 3.4]{audibert2007fast}). This leads to a trivial Bayes decision rule. In the rest of this paper, we rule out such settings by only considering target domains such that $\alpha\beta\le 1$.


Combining all the preceding restrictions, we consider the class $ \cP $ of source and target distribution pairs $ (P,Q) $ in our study of the label shift problem.

\begin{definition}
\label{def:distribution-class} The distribution class
$  \cP(\mu_-, \mu_+, c_\mu, r_\mu, \eps_P, \alpha, C_\alpha, \beta, C_\beta) $ (or $\cP$ in short) is the set of all  distribution pairs $ (P,Q) $ which satisfies the following conditions:
\begin{enumerate}
\item $ P(\cdot | Y) = Q(\cdot | Y) $ (label shift);
\item $ Q_X $ satisfies the strong density condition with parameters $ \mu = (\mu_-, \mu_+)$, $c_\mu>0$, $r_\mu>0$ (see Definition \ref{def:strong-density});
\item $\eps_P \le \pi_P \le 1-\eps_P$ for some $\eps_P > 0$;
\item $ g_0 $ and $ g_1$ are  locally $ \beta $-H\"older smooth (see Definition \ref{def:holder-smooth});
\item $ \eta_Q $ satisfies the margin condition (see Definition \ref{def:margin-condition});
\item $ \alpha\beta \le 1$, where $\alpha$ is the parameter of the margin condition and $\beta$ is the H\"older smoothness parameter.
\end{enumerate}
\end{definition} 


The goal of the learner is to learn a decision rule $ \hat f $ from all the available data (including data from both source and target domains) that has small excess risk in the target domain. To study the difficulty of the label shift problem in both supervised and unsupervised settings, we study their minimax risks as functions of the sample sizes in the source and target domains $n_P, n_Q$.

\section{Supervised label shift}
\label{sec:labeled-target}

In the supervised version of the label shift problem, the learner has access to a dataset $ \datalabeled $, which includes $ n_P $ labeled samples from the source domain and $ n_Q $ \emph{labeled} samples from the target domain. We assume the source and target distribution pair $(P,Q)$ is in $ \cP $ (see Definition \eqref{def:distribution-class}). 
First, we present an information-theoretic lower bound on the convergence rate of the excess risk in the supervised label shift problem. This is a lower bound on the performance of all learning algorithms which accept data $\datalabeled$ and return a classifier $f: [0,1]^d \to \{0,1\}$. Formally, such a learning algorithm is a map $\cA : \samplespacelabeled \to \cH$, where $ \samplespacelabeled \triangleq (\cX \times \cY)^{n_P+n_Q} $ is the space of possible datasets in the supervised label shift problem (the subscript `L' indicates the data from the target domain is labeled) and $\cH \triangleq \{ h : [0,1]^d \to\{0,1\} \}$ is the set of all possible classifiers on $[0,1]^d.$

\begin{theorem}[lower bound for the supervised label shift problem]
\label{th:lower-bound}
 There is $ c>0 $ independent of   $ n_P, n_Q $ and $ \eps_P $ such that 
\[ \inf_{\mathcal A : \samplespacelabeled \to  \cH}\left\{ \sup_{(P,Q)\in \cP} \Ex_{\datalabeled} \Big[\cE_Q(\cA (\datalabeled)) \Big]\right\}  \ge c\left( (n_P\eps_P + n_Q)^{-\frac{\beta}{2\beta+d}} + n_Q^{-1/2} \right)^{1+\alpha}. \] 
\end{theorem}

To show that the preceding lower bound is sharp, we consider a simple plug-in classifier whose convergence rate matches the lower bound:
\begin{equation}
\label{eq:plug-in-classfier}
    \hat{f}(x) \triangleq \begin{cases}1 & \text{if $\hat{\eta}_Q(x) \ge \frac12$,} \\
0 & \text{otherwise,}\end{cases}
\end{equation}
where $\hat{\eta}_Q$ is a particular estimator of the regression function. This estimator is constructed by plugging in an estimator of $\pi_Q$ and kernel-based estimators of $g_0$ and $g_1$ in \eqref{eq:regression-function}:
\begin{equation}
\label{eq:labeled-regfn}
\hat \eta _Q(x) = \frac{\hat \pi_Q \hat g_1(x)}{\hat \pi_Q \hat g_1 (x) + (1-\hat \pi_Q) \hat g_0(x)}.
\end{equation}
Here $\hat\pi_Q$ is the fraction of samples from the target domain with label 1, and $\hat g_0$ and $\hat g_1$ are kernel-based density estimators:
\begin{equation}
\hat g_y(x)  = \frac1{n_y} \sum_{x' \in \cX_y}\frac1 {h_y^d} K\Big(\frac{x-x'}{h_y}\Big), \quad y\in\{0,1\},
\label{eq:density-estimate}
\end{equation}
where $ \cX_y = \{ x: (x,y')\in \datalabeled,\ y'  =y \} $, $ n_y = |\cX_y| $, and $h_y > 0$ is a bandwidth parameter. 
The kernel $K$ in \eqref{eq:density-estimate} is a $\beta^*$-valid kernel \cite[\S 1.2]{tsybakov2009Introduction} for some $\beta^* \ge \beta$. Recall $K$ is a $\beta^*$-valid kernel iff
\begin{equation}
\int K(x)dx = 1,\quad\int x^lK(x)dx = 0\text{ for all }l\in[\beta^*].
\label{eq:valid-kernel}
\end{equation}
This choice of kernel is motivated by the H\"older smoothness assumption on $g_0$ and $g_1$. We note that \eqref{eq:density-estimate} uses samples from both source and target domains to estimate $g_0$ and $g_1$; this does not cause bias in $\hat g_0$ and $\hat g_1$ because in the label shift problem, the class conditionals are identical in the source and target domains. 

\begin{theorem}[upper bound for the supervised label shift problem]
\label{th:upper-bound}
Let $ \hat f $ be the plug in classifier defined above with bandwidths $ h_y =  n_y^{-1/(2\beta + d)}$, $y\in\{0,1\}$, where $n_y$ is the total number of samples in $\cD_L$ that has label $y$.   There is $C > 0$ independent of $n_P, n_Q$ and $\eps_P$ such that  \[ \sup_{(P,Q)\in \cP} \Ex_{\datalabeled} \left[\cE_Q(\hat f)\right] \le C \left( (\eps_P n_P + n_Q )^{-\frac{\beta}{2\beta+d}} + n_Q^{-1/2} \right)^{1+\alpha}. \] 
\end{theorem}


\begin{remark}
The bandwidths $ h_0 $ and $h_1$ and $\beta^*$ (in the choice of kernel) in Theorem \ref{th:lower-bound} depend on the smoothness parameter $ \beta$. In practice, $ \beta $ is usually unknown, so $h_0, h_1$ and $\beta^*$ are chosen by cross-validation. 
\end{remark}

We defer the proofs of Theorems \ref{th:lower-bound} and \ref{th:upper-bound} to the supplement. Together, the two theorems imply the minimax rate of the excess risk is: 
\begin{equation}
\inf_{\cA : \samplespacelabeled \to  \cH} \left\{ \sup_{(P,Q)\in \cP} \Ex_{\datalabeled} \Big[\cE_Q(\cA (\datalabeled)) \Big] \right\}  \asymp\left( (\eps_P n_P + n_Q )^{-\frac{\beta}{2\beta+d}} + n_Q^{-1/2} \right)^{1+\alpha} .
\label{eq:minimax-rate}
\end{equation}
The minimax rate shows the benefits of transfer learning, especially when $n_P \gg n_Q$. In the IID setting in which the learner has access to samples from the target domain but not the source domain, the minimax rate simplifies to \[ \inf_{\cA : \samplespacelabeled \to  \cH} \left\{ \sup_{(P,Q)\in \cP} \Ex_{\cD\sim Q^{\otimes n_Q} } \left[\cE_Q(\cA(\cD))\right] \right\}  \asymp n_Q^{-\frac{\beta (1+\alpha)}{2\beta + d}} \] (recall $\beta/(2\beta + d) < 1/2$). This agrees with known results on the hardness of non-parametric classification in IID settings \citep{audibert2007fast}. This is also the minimax rate of learners who ignore the data from the source domain, To see the benefits of transferring knowledge from the source domain, let $n_P\gg n_Q$. As long as $\eps_P n_P < n_Q^{1 + \frac{d}{2\beta}} - n_Q$ ($\frac{d}{2\beta}$ can be large, so this does not conflict with $n_P\gg n_Q$), the minimax rate simplifies to $(\eps_P n_P + n_Q )^{-{\beta(1+\alpha)}/{(2\beta+d)}}$, which is faster than the minimax rate of learners who ignores the data from the source domain. We wrap up this section with some technical remarks about the minimax rate in \eqref{eq:minimax-rate}.

\begin{remark}
The first term in \eqref{eq:minimax-rate} depends on the hardness of estimating the class conditional densities $g_0$ and $g_1$. This term depends on the \emph{total} sample size $\eps_P n_P + n_Q$ from the source and target domains because samples from both domains are  equally informative in estimating $g_0$ and $g_1$. The astute reader may wonder why $\eps_P$ affects the total sample size but $\pi_Q$ does not. This is because the hardest problem instance has $\pi_P$ close to zero and $\pi_Q = \frac12$. Thus class imbalance in the target domain does not affect the minimax rate, while that in the source domain heavily does. An intuitive explanation for  $\pi_Q = \frac12$ being the hardest case of classification can be understood from the convergence rate of $\hat{\eta}_Q$:
\[
 r(n, \pi)= \left\{ \begin{aligned}
& \sqrt{\frac{(1-\pi_Q)\pi_Q}{n_Q}} +  \sqrt{\pi_Q} \Big( \pi_P n_P + \pi_Q n_Q \Big)^{- \frac{\beta}{2\beta + d}} \\ & +  \sqrt{ (1-\pi_Q)} \Big( (1-\pi_P) n_P + (1-\pi_Q) n_Q \Big)^{- \frac{\beta}{2\beta + d}} 
\end{aligned}\right..
\]
Inspecting $r(n, \pi)$ as a function of $\pi_Q$ one can see that the function achieves maximum rate when $\pi_Q$ is close to $\frac12$. We refer to the proof of the upper bound in Appendix \ref{proof:upper-bound-labeled} (especially the discussion around \eqref{eq:reg-fn-rate-labeled}) for more details.
\end{remark}

\begin{remark}
The exponent of $n_P\eps_P + n_Q $ in \eqref{eq:minimax-rate} depends on the smoothness of $g_0$ and $g_1$; similar exponents arise in the minimax rates of density estimation \citep{tsybakov2009Introduction} and density ratio estimation \citep{kpotufe2017Lipschitz}. The second term in the minimax rate depends on the hardness of estimating  $\pi_Q$. Finally, the overall exponent on the outside depends on the noise level, which we measure with the parameters of the margin condition. We wrap up a few additional remarks about the minimax rate in the supervised label shift problem.
\end{remark}

\begin{remark}
If the learner only has access to the labels, but the access  to the features from the target domain are restricted, then it is possible to adapt the proofs of Theorems \ref{th:upper-bound} and \ref{th:lower-bound} to show that the minimax rate is \[ \inf_{\cA : \samplespacelabeled \to  \cH}\left\{ \sup_{(P,Q)\in \Pi} \Ex \left[\cE_Q(\hat f)\right]  \right\} \asymp \left( (\eps_P n_P  )^{-\frac{\beta}{2\beta+d}} + n_Q^{-1/2} \right)^{1+\alpha}. \] 
At a high-level, the absence of features from $Q$ prevents us from using the samples from $Q$ for estimating the class conditional densities (but they are still useful for estimating $\pi_Q$). 
\end{remark}

\begin{remark}
Unlike similar results on the fundamental limits of other transfer learning settings (\eg\ covariate shift \citep{kpotufe2018Marginal}, posterior drift \citep{cai2019Transfer}, \etc), in the label shift problem, there is a sharp jump between the $\pi_P = \pi_Q$ ($P=Q$) regime and the $\pi_P\ne\pi_Q$ ($P\ne Q$) regime. This is most clearly seen by considering the case in which $P$ is known and the only unknown quantity is $\pi_Q$, which simplifies the problem to a Bernoulli proportion estimation problem. In general, it is not possible to estimate $\pi_Q$ at a rate faster than $\frac{1}{\sqrt{n_Q}}$ unless we know $\gamma \triangleq |\pi_P - \pi_Q| \sim o(\frac{1}{\sqrt{n_Q}})$. In this case, (because we know $P$ and hence we also know $\pi_P$), we can simply estimate $\pi_Q$ with $\widehat{\pi}_Q = \pi_P$, and the worst-case risk of this estimator (over the problem class in which $\gamma \sim o(\frac{1}{\sqrt{n_Q}})$) is $o(\frac{1}{\sqrt{n_Q}})$. Consider the other problem class in which $\gamma\sim\Omega(\frac{1}{\sqrt{n_Q}})$. In this case, Le Cam's method shows that the minimax rate is at least $\frac{1}{\sqrt{n_Q}}$ (the two point construction uses two Bernoulli distributions whose means are $\Theta(\frac{1}{\sqrt{n}})$ apart). 

Taking a step back to the full problem (in which $P$ is unknown), the dependence of the minimax rate on $\gamma \triangleq |\pi_P-\pi_Q|$ is (by analogy to the simple case in which $P$ is known)
\[
\begin{cases}
\frac{1}{\sqrt{n_Q}} + \big( \eps_P n_P +  n_Q \big)^{-\frac{\beta}{2\beta + d}}, & \gamma \sim \Omega(\frac{1}{\sqrt{n_Q}}), \\
\max\{\gamma,\frac{1}{\sqrt{n_P}}\} + \big( \eps_P n_P +  n_Q \big)^{-\frac{\beta}{2\beta + d}}, & \gamma \sim o(\frac{1}{\sqrt{n_Q}}).
\end{cases}
\]
We see that as soon as $\gamma\sim\Omega(\frac{1}{\sqrt{n_Q}})$, the minimax rate no longer depends on $\gamma$. However, the problem sub-class in which $\gamma$ is small  consists of source and target distributions that are very similar, so this sub-class is uninteresting from a transfer learning perspective.

In similar studies of minimax rates for covariate shift \citep{kpotufe2018Marginal} and posterior drift \citep{cai2019Transfer}, the minimax rate depends smoothly on a parameter $\gamma$ that quantifies the allowable difference between the source and target distributions in the problem class, but this does not occur in the minimax rate for the label shift problem. The Bayes decision rule is \emph{different} in the source and target domains in the label shift setting, so optimal prediction (in the target domain) entails estimating an \emph{additional} correction term to bridge the gap between the source and target domains. Estimating this correction term basically boils down to estimating $\pi_Q$, and the hardness of estimating $\pi_Q$ does not depend on the difference between the source and target distributions measured in terms of $\gamma \triangleq |\pi_P-\pi_Q|$ (except in a very small problem sub-class in which transfer learning is irrelevant). Thus our minimax rate for the label shift problem does not depend on $\gamma$.

\end{remark}

\section{Unsupervised label shift}
\label{sec:unlabeled-target-data}
In the unsupervised version of the label shift problem, the learner has access to $ \dataunlabeled $, which consists of  $ n_P $ labeled samples from source domain and $ n_Q $ unlabeled samples from the target domain. For this version of the label shift problem, we impose an extra separation condition between the class conditional distributions. As we shall see, a preliminary step in solving the unsupervised label shift problem is estimating $\pi_Q = Q(Y=1)$, and separation between the class conditionals is necessary for its accurate estimation. We consider the $L_2$-distance $D$ between probability densities:
\begin{equation}
D^2(g_0, g_1) \triangleq \int_{\cX}  \big(g_1(x) - g_0(x)\big) ^2dx\,.
\label{eq:kernel-distance}
\end{equation} For some $0<C <1$ we assume 
$ g_0 $ and $ g_1 $ satisfy $D(g_0,g_1) \ge C$.
In the rest of this section, we work with the following class of source and target distribution pairs $(P,Q)$: \[ \cP' \triangleq \left\{ (P, Q) \in \cP: D(g_0,g_1) \ge C \right\} \]

First, we present a lower bound for the convergence rate of the excess risk in the unsupervised label shift problem. 
The lower bound is valid for any learning algorithm $\cA:\samplespaceunlabeled \to \cH$, where $\samplespaceunlabeled \triangleq (\cX \times \cY)^{n_P}\times \cX^{n_Q}$ is the space of possible datasets in the unsupervised label shift problem (the subscript `U' indicates samples from the target domain are unlabeled) and $\cH \triangleq \{ h : [0,1]^d \to\{0,1\} \}$ is the set of classifiers on $[0,1]^d.$


\begin{theorem}[lower bound for the unsupervised label shift problem]
\label{th:lower-bound-unlabeled}
There is $c > 0$ independent of $n_P, n_Q$ and $\eps_P$ such that  \[ \inf_{\cA :  \samplespaceunlabeled \to \cH }\left\{ \sup_{(P,Q)\in \cP'} \Ex_{\dataunlabeled} \left[\cE_Q\left( \cA \left( \dataunlabeled  \right)\right)\right]\right\}   \ge c \Big( (\eps_Pn_P)^{-\frac{\beta}{2\beta+d}} + n_Q^{-1/2} \Big)^{1+\alpha}. \] 
\end{theorem}

To show that the preceding lower bound is sharp, we design a classifier whose rate of convergence matches the lower bound. As we alluded to earlier, a preliminary step is estimating $\pi_Q$. This is known as the class proportion estimation problem, and it is challenging because the samples from the target domain are unlabeled (in the unsupervised label shift problem). There are many ways to solve the class proportion estimation problem \citep{lipton2018Detecting,azizzadenesheli2019Regularized, alexandari2020EM, du2014semi, iyer2014Maximum, jain2016Estimating}. To prove a matching upper bound, we appeal to the method of \citet{iyer2014Maximum}, but it is possible to show similar upper bounds with other methods (see Remark \ref{rem:alg-class-proportion-estimation}).
Armed with an estimate of $\pi_Q$, we consider the same plug-in classifer as \eqref{eq:labeled-regfn}, except that the kernel-based estimators of $g_0$ and $g_1$ only depend on the labeled samples from the source domain. 

\begin{theorem}[upper bound for the unsupervised label shift]
\label{th:upper-bound-unlabeled}
There is $C > 0$ independent of $n_P, n_Q$ and $\eps_P$ such that \[ \sup_{(P,Q)\in \cP'} \Ex_{\dataunlabeled} \left[\cE_Q \left( \hat f
\right)\right] \le C \left( (\eps_P n_P)^{-\frac{\beta}{2\beta+d}} + n_Q^{-1/2} \right)^{1+\alpha}. \]
\end{theorem}
We defer the proofs of Theorems \ref{th:lower-bound-unlabeled} and  \ref{th:upper-bound-unlabeled} to Section \ref{sec:lb-proof} and Appendix \ref{sec:ub-proofs}. Theorems \ref{th:lower-bound-unlabeled} and  \ref{th:upper-bound-unlabeled} together show that the minimax convergence rate of the excess risk in the unsupervised version of the label shift problem is  
\begin{equation}
\inf_{\cA :  \samplespaceunlabeled \to \cH}\left\{ \sup_{(P,Q)\in \cP'} \Ex_{\dataunlabeled} \left[\cE_Q\left( \cA \left( \dataunlabeled  \right)\right)\right]\right\}   \asymp \left( (\eps_P n_P)^{-\frac{\beta}{2\beta+d}} + n_Q^{-1/2} \right)^{1+\alpha}. 
\label{eq:minimax-rate-unsupervised-label-shift}
\end{equation} 


Recall the minimax rate of IID non-parametric classification in the target domain is $n_Q^{-\frac{\beta(1+\alpha)}{2\beta + d}}$ \citep{audibert2007fast}. Comparing the minimax rates of IID non-parametric classification (in the target domain) and the unsupervised label shift problem, we see that as long as there are enough samples in the target domain (so we are in the regime of the preceding remark), then labeled samples from the source domain are as informative as labeled samples from the target domain. An extremely important practical implication of this observation is if labeled examples are hard to obtain in the target domain, then it is possible to substitute them with labeled examples in a label shifted source domain. 

Before moving on, we compare the minimax rates in the supervised \eqref{eq:minimax-rate} and unsupervised \eqref{eq:minimax-rate-unsupervised-label-shift} label shift problems. We see that the rates differ in the first term in the parentheses: there is $\eps_Pn_P$ instead of $\eps_Pn_P + n_Q$ in the minimax rate of the unsupervised version. Recall this term depends on the hardness of estimating the class conditional densities $g_0$ and $g_1$. In the supervised version, the samples from the target domain are labeled, so they can be directly used to estimate $g_0$ and $g_1$. In the unsupervised version, the samples from the target domain are unlabeled, so there is no direct way to use them to estimate $g_0$ and $g_1$. There may be indirect ways to leverage the samples from the target domain (\eg\ by imputing their labels), but our results show that such tricks cannot improve the convergence \emph{rate} of the classifier. 
We wrap up with a few additional remarks about the minimax rate in the unsupervised label shift problem. 

\begin{remark}
If $ n_P \gg n_Q$, then the minimax rate simplifies to 
\[ \inf_{\cA } \sup_{(P,Q)\in\cP'} \Ex_{\dataunlabeled} \left[\cE_Q\left( \cA \left( \dataunlabeled  \right)\right)\right]   \asymp \begin{cases}
(\eps_Pn_P)^{-\frac{\beta(1+\alpha)}{2\beta + d}} & \text{if } \eps_P n_P \ll n_Q^{1+\frac d{2\beta}},\\
n_Q ^{-\frac{1+\alpha}{2}} & \text{if } \eps_P n_P \gg n_Q^{1+\frac d{2\beta}}.
\end{cases} \]
Recalling the form of the plug-in classifier, we see that there are  two main sources of errors that contribute to the excess risk: 
\begin{enumerate}
\item error in the estimation of the class probability $\pi_Q$. This leads to the $\cO(n_Q ^{-\frac{1+\alpha}{2}})$ term in the minimax rate.
\item error in the estimation of the class conditional densities $ g_0 $ and $ g_1$. This leads to the $\cO((\eps_Pn_P)^{-\frac{\beta(1+\alpha)}{2\beta + d}})$ term in the minimax rate.
\end{enumerate}
If $ \eps_P n_P \gg n_Q^{1+\frac d{2\beta}} $ then the error in estimation of $ \pi_Q$  dominates the excess risk. In this case, improving the estimates of the class conditional densities (\eg\ by increasing $ n_P $) does not improve the overall convergence rate.
\end{remark}
\begin{remark}
If $ \eps_P n_P \ll n_Q^{1+\frac d{2\beta}}, $ then the minimax rate simplifies to \[ \inf_{\cA :  \samplespaceunlabeled \to \cH} \left\{ \sup_{(P,Q)\in \cP'} \Ex_{\dataunlabeled} \left[\cE_Q\left( \cA \left( \dataunlabeled  \right)\right)\right] \right\}  \asymp (\eps_P n_P)^{-\frac{\beta(1+\alpha)}{2\beta + d}}, \]
which is the minimax rate of IID non-parametric classification in the source domain. In other words, given enough unlabeled samples from the target distribution, the error in the non-parametric parts of the unsupervised label shift problem dominate. As this is also the essential difficulty in the IID classification problem in the source domain, it is unsurprising that the minimax rates coincide.
\end{remark}

\begin{remark}
\label{rem:alg-class-proportion-estimation}
Reviewing the methods for class proportion estimation shows that most methods converge at a $n_Q^{-1/2} + \delta(n_P)$-rate uniformly on $\cP'$, where $\delta(n_P)$ is a term that captures the dependence of the rate on $n_P$. Inspecting the proof of Theorem \ref{th:upper-bound-unlabeled} reveals that as long as 
$\delta(n_P) \lesssim (\eps_Pn_P)^{-\beta/(2\beta+d)}$ (which is satisfied by the method of \citet{iyer2014Maximum}), the excess risk of the resulting classifier attains the minimax rate \eqref{eq:minimax-rate-unsupervised-label-shift}. 
Thus, it is possible to prove similar upper bounds with other methods for class proportion estimation as well \citep{lipton2018Detecting,azizzadenesheli2019Regularized, alexandari2020EM, du2014semi, jain2016Estimating}. 
\end{remark}

\section{Simulations}
\label{sec:simulations}

In this section, we present simulations that illustrate the effects of class imbalance in the source domain. The labels in the source domain are distributed as $Y_i \sim \Ber(\pi_P)$, and the labels in the target domain are distributed as  $Y_i\sim\Ber(0.75)$. The class conditional distributions (in both source and target domains) are
\[\textstyle
X_i\mid Y_i \sim Y_i * \text{TN}(0,1, -2, 2)^{\otimes 3} + (1-Y_i) * \text{TN}\left(2,1, 0, 4\right)^{\otimes 3},
\]
where $\TN(\mu, \sigma^2, a,b) $ is the $N(\mu,\sigma^2)$ distribution truncated to the interval $[a,b]$. We consider three class imbalance settings: $\pi_P = 0.5$ (solid line, solid circle as pointer), $\pi_P \sim 1/\sqrt{n_P}$ (dashed line, star as pointer) and $\pi_P \sim 1/n_P$ (dotted line, plus as pointer).
We defer other details of the simulation setup to Appendix A. 

\paragraph{Supervised label shift simulations} In the supervised setting, we compare the excess risks of the two following methods:
\begin{enumerate}
\item a minimax rate optimal plug-in classifier $ \indicator\left\{ \hat \eta_Q(x) \ge \frac12  \right\}$,  where $\hat \eta_Q$ is defined in \eqref{eq:labeled-regfn}. This method is denoted as \texttt{Cl-labeled}.
\item  a classical classifier designed for IID settings. This is the same classifier as the minimax rate optimal plug-in classifier that only uses data from the \emph{target} domain to estimate the class conditional densities and class probabilities. We considered this classifier in Section \ref{sec:labeled-target} when discussing the benefits of transfer learning. This method is denoted as \texttt{Cl-classical}.
\end{enumerate}
To estimate the class conditional densities, we use a $2$-valid kernel (see \eqref{eq:valid-kernel}) with the optimal bandwidth $h_0 ={n_0^{-1/7}} $ and $ h_1 = {n_1^{-1/7}}$ (see Theorem \ref{th:upper-bound}). 


We compare the efficacy of the two approaches in two settings.
In the first setting (Figure \ref{fig:supervised}, left), $n_P \gg n_Q$ ($n_P$ is growing and $n_Q$ is held fixed), so it is necessary to transfer knowledge from source to target data to reduce the excess risk. 
As expected, the excess risk of the classical classifier designed for IID settings remains constant as $n_P$ changes as it ignores data from the source domain. On the other hand, the excess risk of the \texttt{Cl-labeled} classifier, which transfers knowledge from source to target domain, decreases as $n_P$ increases in the $\pi_P = 0.5$ and $\pi_P \sim \frac1 {\sqrt{n_P}}$ settings. This shows that \texttt{Cl-labeled} leverages data from the source domain to improve classification accuracy in the target domain. We note that effect of class imbalance in the source domain appears in the $\eps_Pn_P$ term in the minimax rate \eqref{eq:minimax-rate}. Recall this is the minimum expected per class sample size  in the source domain. From \eqref{eq:minimax-rate}, we expect the convergence rate of the excess risk is slower in the $\pi_P \sim \frac1 {\sqrt{n_P}}$ setting than in the $\pi_P = 0.5$ setting, and we see that this is indeed the case in Figure \ref{fig:supervised}. We also expect the excess risk to remain bounded away from zero in the $\pi_P \sim \frac 1{n_P}$ setting because the $\eps_Pn_P$ term remains bounded away from zero in the minimax rate.

In the second setting (Figure \ref{fig:supervised}, right), we compare the classical classifier and \texttt{Cl-labeled} in the $n_Q \gg n_P$ ($n_Q$ is growing and $n_P$ is held fixed) setting. This is an easier setting because it is possible to perform well (have small excess risk) in the target domain without transferring knowledge from the source domain. We expect both approaches to perform well in this setting, and Figure \ref{fig:supervised} confirms this.


\begin{figure}
\centering
\includegraphics[height=5cm,width=12cm]{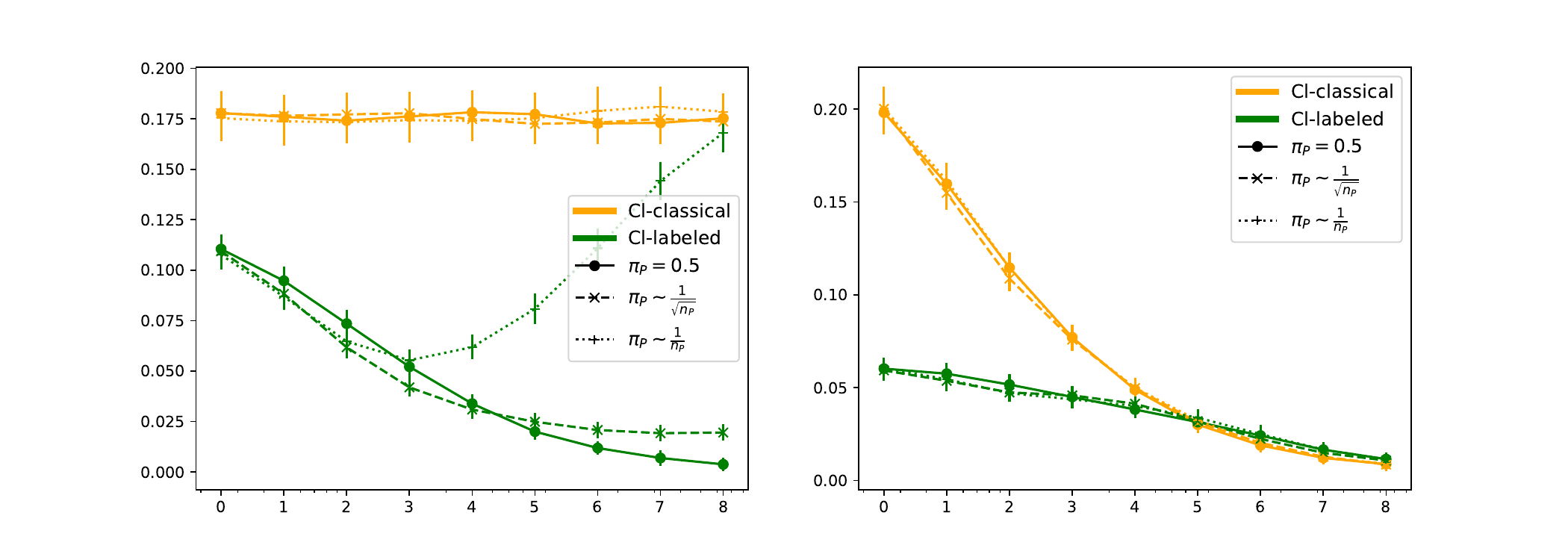} 

\hspace{-8em} 

\begin{picture}(0,18)(0,0)
\put(-155, 90){\rotatebox{90}{  $\widehat \cE (f)$}}
\put(-125, 25){$k,\ n_P = 25 \times 2^k;$ \textcolor{blue}{$n_Q = 40$}}
\put(25, 25){$k,\ n_Q = 25 \times 2^k;$ \textcolor{blue}{$n_P = 160$}}
\end{picture}
\caption{Excess risks in the supervised label shift problem. We see that the minimax optimal approach studied in section \ref{sec:labeled-target} is the only approach that learns effectively from both source and target domains.}
\label{fig:supervised}
\end{figure}
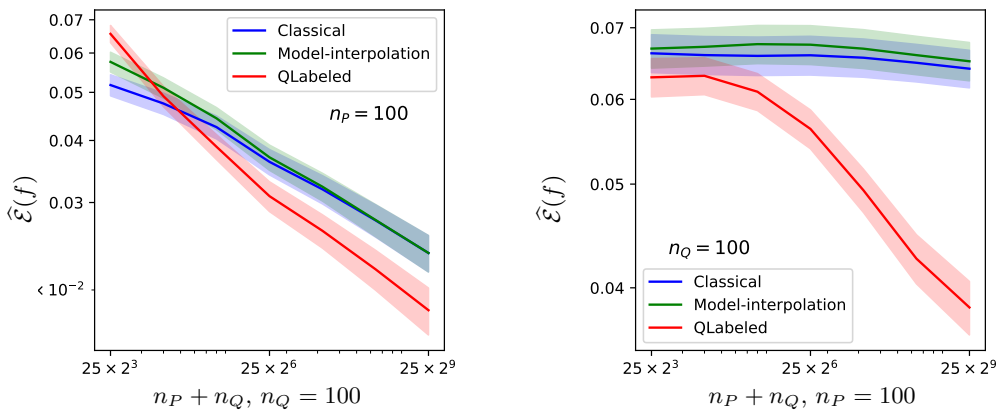


\paragraph{Unsupervised label shift simulations} In the unsupervised setting, we compare the excess risks of the two following methods:
\begin{enumerate}
    \item a minimax rate optimal plug-in classifier where $\hat\eta_Q$ is estimated using kernel-based density estimates of $g_1$ and $g_0$ and a distribution matching estimator of $\pi_Q$ \citet{lipton2018Detecting}. This method is denoted as \texttt{Cl-unlabeled}. We use a different method here to estimate $\pi_Q$ than in the proof of Theorem \ref{th:upper-bound-unlabeled} to demonstrate the robustness of our results to the choice of estimator (of $\pi_Q$). 
    \item an oracle plugin classifier that uses the exact value of $\pi_Q$. We denote this classifier by \texttt{Cl-oracle}. 
\end{enumerate}
To estimate the class conditional densities, we use a $3$-valid kernel (see \eqref{eq:valid-kernel}) with the optimal bandwidths $h_0 = {\big(n_0'\big)^{-1/7}},\ h_1 =  {\big(n_1'\big)^{-1/7}}$ (see theorem \ref{th:upper-bound-unlabeled}). We defer the other details of the simulation setup to Appendix A. 

We evaluate the efficacy of both approaches in the two preceding settings, $n_P \gg n_Q$ (see Figure \ref{fig:unsupervised}, left) and $n_Q \gg n_P$ (see Figure \ref{fig:unsupervised}, right). In both settings, we fix $\pi_Q = 0.75$ and vary $\pi_P$ to study the effects of class imbalance. 
At a high-level, as long as $n_Q$ is large enough, the distributional matching classifier matches the performance of the oracle classifier. This is explained by the fact that if $n_Q$ is large enough, the distributional matching produces a good enough estimator of $\pi_Q$, so the error in the estimation of $\pi_Q$ is no longer the dominant source of error in the excess risk. 

We observe that there is always a small gap between the excess risk of the two methods because the estimates of the class probability ratios are not consistent in the simulation settings. Recall that the error incurred by distributional matching is $O(\frac{1}{n_P}) \vee O(\frac{1}{n_Q})$ (see \cite{lipton2018Detecting}, Theorem 3). This implies that the estimates of the class probability ratios never converge to their population counterparts in the simulation settings, because one of $n_P$ and $n_Q$ is held fixed by design. Thus, the oracle classifier, which does not suffer from errors in the estimates of the class probability ratios, is always a step ahead of the minimax rate optimal classifier.

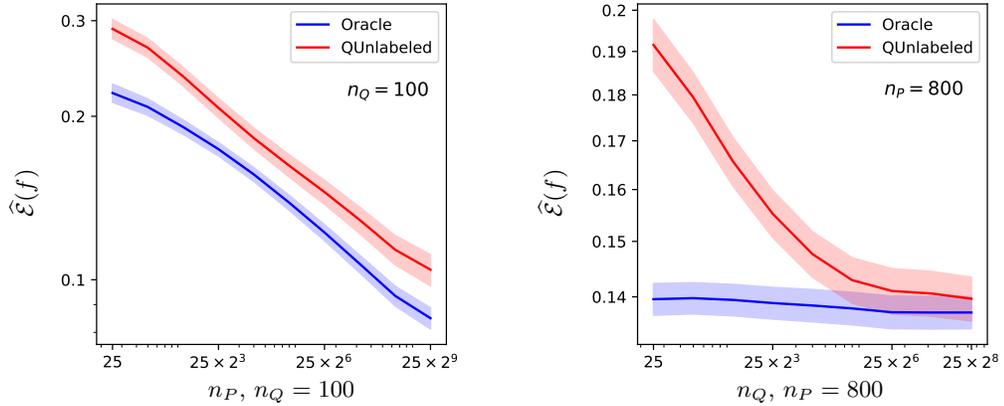
\begin{figure}
\centering
\includegraphics[height=5cm,width=12cm]{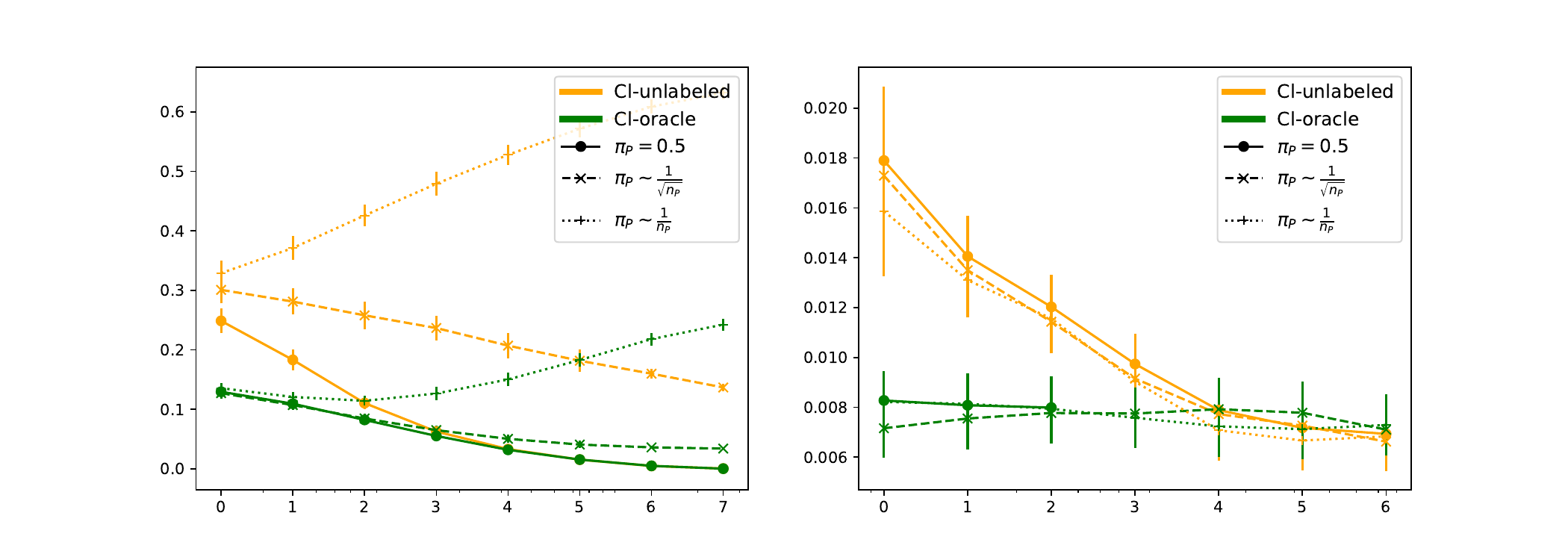} 
\hspace{-8em} 
\begin{picture}(0,18)(0,0)
\put(-255, 60){\rotatebox{90}{  $\widehat \cE (f)$}}
\put(-225, 0){$k,\ n_P = 20 \times 2^k;$ \textcolor{blue}{$n_Q = 100$}}
\put(-75, 0){$k,\ n_Q = 10 \times 2^k;$ \textcolor{blue}{$n_P = 1000$}}
\end{picture}
\caption{Excess risks in the unsupervised label shift problem. We see that the minimax optimal approach studied in section \ref{sec:unlabeled-target-data} (eventually) matches the performance of the oracle classifier.}
\label{fig:unsupervised}
\end{figure}

\section{Proof of Theorem \ref{th:lower-bound-unlabeled}}
\label{sec:lb-proof}

In this section, we prove Theorem \ref{th:lower-bound-unlabeled}. The proof of Theorem \ref{th:lower-bound} is similar, and we defer the details to the supplementary materials. At a high-level, the proof has two parts. The first part shows that the minimax rate is at least $(\eps_P n_P)^{-\frac{\beta(1+\alpha)}{2\beta + d}}$. This is due to the difficulty arising from  the non-parametric part of the label shift problem: estimating the class conditional densities $g_0$ and $g_1$. The second part shows that the minimax rate is at least $n_Q^{- \frac{1+\alpha}{2}} $. This stems from  the difficulty of the parametric part of the label shift problem: estimating the class probabilities in target domain $\pi_Q$. 

\subsection{Difficulty of the non-parametric part} To study the difficulty of the non-parametric part of the label shift problem, we appeal to the following proposition  to obtain lower bounds in the non-parametric regression problems. 

\begin{proposition}[Theorem 2.5 in  \cite{tsybakov2009Introduction}] 
\label{prop:tsybakov}
Let $ \{ \Pi_h\}_{h\in \cH} $ be a family of distributions indexed over a subset $ \cH $ of a semi-metric $ (\cF,\bar \rho). $ Suppose there are $h_0, \ldots h_M\in \cH , $ $ M\ge 2, $ such that:
\begin{enumerate}
\item $ \bar\rho (h_i , h_j) \ge 2s>0$ for all $0 \le i<j \le M, $
\item $ \Pi_{h_i} \ll \Pi_{h_0} $ for all $ i\in [M] $, and the average KL-divergence to $ \Pi_{h_0} $ satisfies \[ \frac1M \sum_{i=1}^M\KL(\Pi_{h_i}\mid\Pi_{h_0}) \le \kappa \log M\text{ for some }0<\kappa < \frac18. \]
\end{enumerate}
Let $ Z\sim \Pi_h$ and $ \hat f : Z \to \cF $ be any improper learner of $ h \in \cH. $ We have for any $ \hat f : $ \[ \sup_{h\in\cH}  \Pi_h\left( \bar\rho (\hat f (Z),h)  \ge s \right)  \ge \frac {3-2\sqrt{2}}{8}.  \]
\end{proposition}

The crux of this part  is  construction of a family of distributions on the source and target domains $\{\Pi_i\}_{i=1}^M$ that satisfies the assumptions of proposition \ref{prop:tsybakov}. Before delving into the technical details, we describe the intuition behind the construction. 

We devote most of our efforts in this part to constructing the class conditional densities because we wish to study (the difficulty of) the non-parametric part of the problem. We partition the sample space $[0,1]^d$ into small (hyper-)rectangles and divide the rectangles into three groups. The class conditional densities differ on the second and third groups, but they are identical on the first group. The sizes of the two groups are carefully adjusted to satisfy the margin condition in definition \ref{def:distribution-class} and minimize the KL divergence between the $\Pi_i$'s. Within each rectangle, the class conditional densities are smooth  functions with a rise/fall near center of the rectangle and half near the boundaries.


In the proof, we associate the rectangles with the vertices of a hypercube and judiciously pick a subset of rectangles on which the class conditional densities differ, so that the Hamming distance between the associated verticies of the hypercube is maximized. The lower bound on the minimax rate then follows directly from proposition \ref{prop:tsybakov}.

We assume $\eps_Pn_P \ge 1. $
Let  $r = c_r (\eps_Pn_P)^{-1/(2\beta + d)}, m = \lfloor c_m r ^{\alpha\beta - d}\rfloor,$ where $\alpha\ge 0 $ is the noise condition exponent and $\beta>0 $ is the H\"older smoothness exponent.
Also, define  $c_r = (1/17) ,\  c_m = 8 \times 17^{\alpha\beta - d}$ and $c_w \in (0, 1)$ is a constant to be picked later. 
The preceding constants satisfy \[ 8 \le m <\frac12 \left\lfloor \frac 1r \right\rfloor^d. \]  
To see the first inequality, we observe that $r \le 1/17$ and recall $\alpha\beta \le   d.$ 
Thus \[ m = 8 \times (17r)^{\alpha\beta - d} \ge 8. \] 
To see the second inequality, we observe that $r^{-1} \ge 16,$ which implies $r^{-1}\le 17\lfloor r^{-1}\rfloor/16.$ 
Thus \[ m = 8 (17r)^{\alpha\beta} \left(\frac 1{17r}\right)^d \le 8\cdot  16^{-d} \lfloor r^{-1} \rfloor ^d \le\frac1 2 \lfloor r^{-1} \rfloor ^d. \] We also have $2mw = 2mc_w r^d\le 2c_m c_w<1 $ for a suitable $c_w.$
    
    \textit{Construction of $\{\Pi_i\}_{i=1}^M$.}
    Let $r_1 = 1/\lfloor 1/(c_wr)\rfloor$ if $ \lfloor 1/(c_wr)\rfloor $ is even, otherwise let $ r_1 = 1/\left(\lfloor 1/(c_wr)\rfloor+1\right). $
      Let us consider the grid of points \begin{equation}
        \cZ = \{ (1/2+i)r_1 : i = 0, 1, \dots , 1/r_1 -1 \}^d.
    \end{equation} 
    We see that $\cZ$ is a grid of equally spaced points of size $r_1^{-d}.$ For a $z \in \cZ$ we consider the hyper-cube \[ C(z) = \{x \in [0,1]^d: \|x-z\|_\infty\le r_1/2 \}. \] Note that, volume of each of these hyper-cubes is $r_1^d.$ Let $\cZ_1, \cZ_2 \subset \cZ $ be subsets of size $m$. Moreover, we let $\cZ_1 $ and $\cZ_2$ are disjoint. We define a bijection $u : \cZ_1 \to \cZ_2 $ which shall be used to construct the conditional densities. We define $\cZ_0 = \cZ \backslash (\cZ_1 \cup \cZ_2).$  Note that, $ \cZ $ has even number of points and $ |\cZ_1| = |\cZ_2| .$ Hence,  $ \cZ_0 $ has even number of points. We further divide $ \cZ_0 $ in two sets $ \cZ_3, \cZ_4 $ of equal sizes.  We shall define a set of distributions parametrized by $\sigma\in \{ -1, 1 \}^{\cZ_1}.$ 

    \noindent\textit{Conditional densities.}  For $a >0$ we define a function $v_a$ supported on on $\reals$ which will be used heavily for the construction of conditional densities. 
 Define   
\[u_{a}(x)=     
\begin{cases}
0 & \text{ for }x<0\\
\frac{ \int_{0}^x e^{-\frac1{at(1-t)}}dt}{\int_{0}^1 e^{-\frac1{at(1-t)}}dt} & \text{ for }0 \le x \le 1\\
1 & \text{ for } x>1
\end{cases}\]
   and 
    \begin{equation}
    \label{eq:u-function-1}
       v_{a}(x)=     \begin{cases}    
    \big( 1-u_a(x) \big)^{1/\alpha}  & \quad \text{for }\beta < 1,\\
    \big( 1 -u_a(x) \big) & \quad \text{for }\beta \ge 1.\end{cases}
    \end{equation}

    According to lemma C.6 we choose $ a $ such that $ v_a \equiv v  $ is $ (\beta, C_\beta) $ H\"older smooth.  
    Therefore, the following functions are $(\beta, C_\beta)$-H\"older smooth: 
    \[ z \in \cZ, \hspace{0.3cm} \eta_z(x) =  \frac{r^\beta}3 v \left(\frac{2\|x-z\|_\infty }{r_1}\right) \]
 and     \[ z \in \cZ, \hspace{0.3cm} \xi_z(x) =   v \Bigg(\frac{2\|x-z\|_\infty }{br_1} - 2\left(\frac1b-1\right)\Bigg). \]
    For a parameter $\sigma$ the  construction of  conditional densities are given below.

    \begin{gather}
        \left\{    
        \begin{aligned}
         g_1^{\sigma}(x) & = \begin{cases}
    1+  \sigma(z) \sqrt{\eps_P} \eta_z(x) & \hspace{0.3cm} x \in C(z),\ z \in \cZ_1,\\
    1-  \sigma(z) \sqrt{\eps_P} \eta_{f(z)}(x) & \hspace{0.3cm} x \in C(f(z)),\  z \in \cZ_1,\\
    1 + \xi_z(x) & \hspace{0.3cm} x \in C(z),\  z \in \cZ_3,\\
    1 - \xi_z(x) & \hspace{0.3cm} x \in C(z),\  z \in \cZ_4,
    \end{cases} \\
    g_0^{\sigma}(x)&  = \begin{cases}
     1-  \sigma(z) \sqrt{\eps_P}  \eta_z(x) & \hspace{0.3cm} x \in C(z),\ z \in \cZ_1,\\
    1+  \sigma(z) \sqrt{\eps_P}  \eta_{f(z)}(x) & \hspace{0.3cm} x \in C(f(z)),\  z \in \cZ_1,\\
    1 - \xi_z(x) & \hspace{0.3cm} x \in C(z),\  z \in \cZ_3,\\
    1 + \xi_z(x) & \hspace{0.3cm} x \in C(z),\  z \in \cZ_4.\end{cases}
        \end{aligned}
        \right.
        \label{eq:density-class-unsupervised-density}
    \end{gather}

    We also define $\pi_Q^{\sigma} = 1/2$ and $\pi_P^\sigma = 1/2.$ We then define the probabilities \[P_{\sigma}(X \in A, Y= y) = \int_A [\pi_P^{\sigma} g^{\sigma}_1(x)\indicator(y = 1) + (1-\pi_P^{\sigma}) g^{\sigma}_0(x) \indicator (y = 0) ] dx\] and
    \begin{equation}
        Q_{\sigma}( X \in A, Y = y) = \int_A [\pi_Q^
    {\sigma}g^{\sigma}_1(x)\indicator(y = 1) + (1-\pi^{\sigma}_Q) g^{\sigma}_0(x) \indicator (y = 0) ] dx.
    \label{eq:dist-Q-unlabeled-density}
    \end{equation}
    Given the source and target distributions we define the joint distribution of $\dataunlabeled$ as \begin{equation}
        \label{eq:dist-unlabeled-joint-density}
        \Pi_\sigma = P_\sigma^{\otimes n_P} \otimes Q_{\sigma, X} ^{\otimes n_Q}
    \end{equation}
    
    Here, $\eps_P \le \pi_P^\sigma \le 1-\eps_P.$ Also, $q_X \equiv 1 $ for any $x \in \Omega.$ Hence, $\mu_- \le q^{\sigma}_X(x) \le  \mu _+ .$ Furthermore, $\Omega = [0,1]^d$ is a regular set. Hence, $q^\sigma _X$ satisfies strong density assumption.

    For such a construction, the marginals are \begin{equation}
        \label{eq:marginals-unlabeled-density}
        p_X^\sigma(x) = q_X^\sigma(x) = \frac 12 g_1^\sigma(x)+ \frac 12 g_0^\sigma(x) = 1
    \end{equation}
    
    Furthermore, the regression function $\eta_Q^\sigma$ is \begin{equation}
        \label{eq:reg-fn-unlabeled-density}
        \eta_Q^\sigma(x) = \frac{\pi_Q^\sigma g_1^\sigma(x)}{q_X^\sigma(x)} = \frac 12 g_1^\sigma(x)
    \end{equation}

   We refer to lemma \ref{lemma:margin-unlabeled-density}, where it is shown $Q_\sigma$ satisfies $\alpha$-margin condition with constant $C_\alpha.$ Also, the separation assumption \ref{assump:separation-conditionals} is verified in lemma \ref{lemma:verification-sepration-conditionals}.

    Let $ \cF $ be the set of all classifier relevant to this classification problem. For $\sigma \in \{-1, 1\}^{\cZ_1}$ let  $f_\sigma$ be the Bayes classifier corresponding to the probability distribution $Q_\sigma$ defined as $ f_\sigma (x) = \indicator\{\eta_{ Q_\sigma }(x) \ge 1/2\}.$  For $ \sigma,\sigma'\in\{-1, 1\}^{\cZ_1} $ define $ \bar\rho (\sigma,\sigma') \coloneqq \cE_{\sigma} (f_{\sigma'})   $ and  $\rho(\sigma,\sigma') = \text{card} \{z\in\cZ_1 :\sigma(z) \neq \sigma'(z) \} $ as the Hamming distance. Then \begin{align*}
        \bar\rho (\sigma,\sigma') & = 2\Ex _{Q_{\sigma, X}} \left[\left | \eta_Q^\sigma(X) - \frac12 \right | \mathbf{1}\left( f_\sigma(X) \neq f_{\sigma'}(X)  \right)\right]\\
         & \ge c_1   r_1^d r^\beta  \rho(\sigma, \sigma')\\
         & \ge c_1 c_w^d r^{\beta+d}\rho(\sigma, \sigma').
    \end{align*}
	
We recall Varshamov-Gilbert bound, which shall be used to construct the probability class. 	
	\begin{lemma}[Varshamov-Gilbert bound]
	\label{lemma:varshamov-gilbert}
	Let $ m\ge 8. $ Then there exists a subset $ \{ \sigma_0,\ldots , \sigma_M \} \subset \{-1,1\}^m $ such that $ \sigma_0 = (1 ,\ldots, 1), $ \[ \rho_H(\sigma_i , \sigma_j ) \ge \frac m 8, \ \text{for all } 0\le i<j\le M,\ \text{and } M \ge 2^{m/8}, \] where, $ \rho_H $ is the hamming distance. 
\end{lemma}

	Let $\{ \sigma_0,\ldots , \sigma_M \} \subset \{-1,1\}^m  $ be the choice obtained from the  lemma \ref{lemma:varshamov-gilbert}.  Note that for such a choice $\rho(\sigma_i, \sigma_j) \ge m/8$ whenever $i\neq j.$
	
	Then  \begin{align*}
	\bar\rho (\sigma_i, \sigma_j) &\ge c_1c_w^d  r^{\beta+d}  \frac m8\\ 
	&\ge c_1c_w^d  r^{\beta+d}   r^{\alpha\beta -d}  \\
	& \ge c' r^{\beta (1+\alpha)}\\
	& = c' (\eps_P n_P)^{-\frac{\beta(1+\alpha)}{2\beta + d}}\\
	& \triangleq 2s
	\end{align*}

	Now we bound the Kulback-Leibler divergence between the joint distributions $\Pi_{\sigma_i}.$ Using lemma \ref{lemma:kl-unlabeled-density} we get 
	\begin{align*}
	KL(\Pi_{\sigma_i}|| \Pi_{\sigma_j}) & \le c_w^dK(d, \alpha, \beta)\rho(\sigma_i, \sigma_j) \\
	& \le  c_w^dK(d, \alpha, \beta)m \\
	& \le \frac 19 \log_2(M) \\
	\end{align*}
	for suitable $c_w < 1.$

	Finally we appeal to proposition \ref{prop:tsybakov} (and Markov's inequality) to obtain the minimax rate \begin{align*}
\sup_{(P,Q)\in \Pi} \Ex \cE_Q (\hat f)& \ge \sup_{(P,Q)\in \Pi} s \Pr_\Pi \left ( \cE_Q (\hat f)\ge s \right)\\
& \ge s \sup_{\sigma\in \{-1,1\}^{\cZ_1}} \Pi_\sigma \left(\cE_{Q^\sigma} (\hat f)\ge s \right) \\
& \ge s \frac {3-2\sqrt{2}}{8}\\
& \ge C(\eps_Pn_P)^{-\frac{\beta(1+\alpha)}{2\beta + d}}.
\end{align*}

\subsection{Difficulty of the parametric part} To study the difficulty of the parametric part, it is enough to construct two well-separated hypotheses (versus the family of well-separated hypotheses required in our study of the non-parametric part). We refer to the following theorem to establish  lower bound. This particular form of LeCam's bound is taken from \cite{tsybakov2009Introduction}, Chapter 2. The result directly follows from Equation (2.5) and statement (iii) in Theorem 2.2.

\begin{theorem}[LeCam's bound]
\label{th:lb-lecam}

Let $\cP$ be a set of distributions. For any pair $P_0, P_1 \in \cP,$
\[ \inf_{\hat \theta }\sup_{P\in \cP} \Ex_P \big[ d(\hat \theta, \theta(P) ) \big] \ge \frac \Delta 8 e^{-KL(P_0||P_1)} \] where $\Delta = d\big(\theta(P_0), \theta (P_1)\big).$

\end{theorem}

The construction will closely follow the non-parametric part. Let $\sigma \in \{-1, 1\},\ c_w<1,\ r = n_Q^{-\alpha/2},$
\[ r_1 = \begin{cases}
1/\lfloor 1/(c_wr) \rfloor & \text{ if }  \lfloor 1/(c_wr) \rfloor \text{ is odd}\\
1/\left(\lfloor 1/(c_wr) \rfloor+1\right) & \text{ if }  \lfloor 1/(c_wr) \rfloor \text{ is even}
\end{cases} \] Thus $ \frac1{r_1} $ is always an odd number. Define $ 2D+1 = 1/r_1. $  Let us consider the grid of points \[
\cZ = \{ z_i = (1/2+i)r_1 : i = 0, 1, \dots ,2 D  \}.\]

We again consider the function $ v_a\equiv v   $ (as defined in \eqref{eq:u-function-1}) which is $ (\beta, C_\beta) $ H\"older smooth. Hence the 
 following functions are also $ (\beta, C_\beta) $ H\"older smooth.      \[ z \in \cZ, \hspace{0.3cm} \eta_z(x) =   v \Bigg(\frac{2\|x-z\|_\infty }{br_1} - 2\left(\frac1b-1\right)\Bigg) \].
We also define $ C(z) = \big\{x \in [0,1]^d: \big| x^{(1)} - z^{(1)} \big| \le r_1/2 \big\} $

\begin{gather}
\label{eq:conditional-parametric}
\left\{  \begin{aligned}
g_1^\sigma(x) & = \begin{cases}
1 & \quad x \in C(z_0), \\
1 + \eta_{z_i}(x) & \quad x \in C(z_i), \text{ for }i = 1, \dots,D , \\
1 - \eta_{z_i}(x) & \quad x \in C(z_i), \text{ for }i = D+1, \dots, 2D ,
\end{cases} \\
g_0^\sigma(x) &= \begin{cases}
1 & \quad x \in C(z_0), \\
1 - \eta_{z_i}(x) & \quad x \in C(z_i), \text{ for }i = 1, \dots,D , \\
1 + \eta_{z_i}(x) & \quad x \in C(z_i), \text{ for }i = D+1, \dots, 2D 
\end{cases} 
\end{aligned}\right.
\end{gather}
and  $\pi_P^\sigma = 1/2,\ \pi_Q^\sigma = 1/2 + \sigma c_Q n_Q^{-1/2} .$ We construct $P_\sigma, Q_\sigma$ and define 
\begin{equation}
\label{eq:dist-unlabeled-joint-parametric}
    \Pi_\sigma = P_{X,\sigma} ^{\otimes n_P} \otimes Q_\sigma ^{\otimes n_Q}
\end{equation}
 in the similar way. As before, we have $\eps_P \le \pi_P^\sigma \le 1-\eps_P$ and $ q_X \equiv  1. $ Hence, $\mu_-\le q_X\le \mu_+$ and supported on $[0, 1]^d.$ This implies $Q_X$ has strong density. Also, we refer to lemma \ref{lemma:margin-unlabeled-density}  to show $Q_\sigma$ satisfies $\alpha$ margin condition with constant $C_\alpha.$

For $\sigma \in \{-1, 1\}$ let $\eta_\sigma$ and $f_\sigma$ be the regression function and Bayes classifier for $Q_\sigma,$ respectively. Then \[
\begin{aligned}
\bar \rho (\Pi_{1}, \Pi_{-1}) &= \cE_{Q_1} (f_{-1}) \\ &= 2 \Ex_{Q_{1, X}} \Big[ \big|\eta_Q^\sigma(X)-1/2\big| \indicator\{ f_1(X) \neq f_{-1}(X) \}  \Big]\\
& \ge 4 c_Qn_Q^{-1/2}r_1  \\
& \ge c' n_Q^{-\frac{1+\alpha}{2}} = 2s.
\end{aligned}
\]

We refer to lemma C.4 in supplement  to get the following bound 
\[ \begin{aligned}
KL(\Pi_1||\Pi_{-1}) & \le n_Q KL(U_Q^{(1)}||U_Q^{(-1)}) \\
& \le n_Q(1/2 + c_Qn_Q^{-1/2}) \log \left(\frac{1/2 + c_Qn_Q^{-1/2}}{1/2 - c_Qn_Q^{-1/2}}\right)\\
& \quad + (1/2 - c_Qn_Q^{-1/2}) \log \left(\frac{1/2 - c_Qn_Q^{-1/2}}{1/2 + c_Qn_Q^{-1/2}}\right)\\
& \le n_Q\frac243c_Q^2n_Q^{-1}\\
& \le c_Q^2
\end{aligned} \] where $ U_Q^\sigma \sim \text{Bernoulli}(\pi_Q^\sigma). $

Using theorem \ref{th:lb-lecam}   we conclude 
\[
    \sup_{(P, Q) \in \Pi} \Ex \big[ \cE_Q(\hat f) \big] \ge c n_Q^{-\frac{1+\alpha}{2}}.
\]

Finally, we combine the two bounds to get \[ 
\begin{aligned}
\sup_{(P, Q) \in \Pi} \Ex \big[ \cE_Q(\hat f) \big] & \ge c n_Q^{-\frac{1+\alpha}{2}} \vee c(\eps_Pn_P )^{-\frac{\beta(1+\alpha)}{2\beta + d}}\\
& \ge c' \Bigg( (\eps_Pn_P)^{-\frac{\beta}{2\beta + d}} + n_Q^{-1/2}  \Bigg)^{1+\alpha} .
\end{aligned}
\]

\begin{lemma}
\label{lemma:margin-unlabeled-density}
For any $\sigma \in \{-1, 1\}^{\cZ_1}$  $Q_\sigma$ as defined in  \ref{eq:dist-Q-unlabeled-density} and \ref{eq:dist-unlabeled-joint-parametric} satisfies $\alpha$-margin condition with constant $C_\alpha.$ 
\end{lemma}

\begin{proof}
	We prove the lemma for both \ref{eq:dist-Q-unlabeled-density} and \ref{eq:dist-unlabeled-joint-parametric}.
	\paragraph{\textbf{Margin condition for non-parametric part}}
We recall the marginal and regression function as in \ref{eq:marginals-unlabeled-density} and  \ref{eq:reg-fn-unlabeled-density}, respectively. For such a regression function \[ 
\big| \eta_Q^{\sigma}(x) - 1/2 \big|  = \begin{cases}
      \sqrt{\eps_P} \eta_z(x)/2 & \hspace{0.3cm} x \in C(z),\ z \in \cZ_1,\\
      \sqrt{\eps_P} \eta_{f(z)}(x)/2 & \hspace{0.3cm} x \in C(f(z)),\  z \in \cZ_1,\\
    0 & \hspace{0.3cm} x \in C(z),\  z \in \cZ_0.
    \end{cases}
\]

Let \[ t_0 = \begin{cases}
\frac{\sqrt{\eps_P}r^\beta}{6} \left(1-u_a(1/2)\right)^{1/\alpha} & \beta < 1\\
\frac{\sqrt{\eps_P}r^\beta}{6} \left(1-u_a(1/2)\right)  & \beta \ge 1
\end{cases} \]

For  $ t \le t_0,  \beta <1  $ and 
 $ z \in \cZ_1 \cup \cZ_2,  $ we see that 
\begin{equation}
\label{eq:ineq-margin-cond}
	\int_{1/2}^1 e^{-\frac1{a(s-s^2)}}ds \le \int_{0}^1 e^{-\frac1{a(s-s^2)}}ds\left(\frac{6t}{\sqrt{\eps_P}r^\beta}\right)^\alpha.
\end{equation} Hence, 
\[ \begin{aligned}
& Q_X\big\{ 0 <\sqrt{\eps_P}\eta_z(x) \le 2t\big\}\\
 & =   Q_X\left\{  0 <  1-u \left(\frac{2\|x-z\|_\infty }{r_1}\right) \le \left(\frac{6t}{\sqrt{\eps_P}r^\beta}\right)^\alpha  \right\}\\
& = Q_X\left\{ 0<\int_{\frac{2\|x-z\|_\infty}{r_1}}^1 e^{-\frac1{a(s-s^2)}}ds \le \int_{0}^1 e^{-\frac1{a(s-s^2)}}ds\left(\frac{6t}{\sqrt{\eps_P}r^\beta}\right)^\alpha   \right\}\\
& \le  Q_X\left\{  0<e^{-\frac4{a}} \left(1-\frac{2\|x-z\|_\infty}{r_1}\right) \le \int_{0}^1 e^{-\frac1{a(s-s^2)}}ds\left(\frac{6t}{\sqrt{\eps_P}r^\beta}\right)^\alpha   \right\}\\
& =   Q_X\left\{  1-e^{4/a} \int_{0}^1 e^{-\frac1{a(s-s^2)}}ds\left(\frac{6t}{\sqrt{\eps_P}r^\beta}\right)^\alpha <\frac{2\|x-z\|_\infty}{r_1}\le 1    \right\}\\
& = r_1^d \left[1- \Bigg(1-e^{4/a} \int_{0}^1 e^{-\frac1{a(s-s^2)}}ds\left(\frac{6t}{\sqrt{\eps_P}r^\beta}\right)^\alpha\Bigg)^d\right]\\
& \le C \eps_P^{-\alpha/2} c_w^d \frac{r^{d}t^\alpha}{r^{\alpha\beta}} 
\end{aligned} \] where the third inequality is true because $ \frac{2\|x-z\|_\infty}{r_1} \ge 1/2 $ (obtained from inequality \ref{eq:ineq-margin-cond}).
Similarly,  $ t \le t_0 , \beta \ge 1 $ we have  \[
\begin{aligned}
Q_X\big\{ 0 <\sqrt{\eps_P}\eta_z(x) \le 2t\big\} &  \le C b^d \eps_P^{-1/2} c_w^d \frac{r^{d}t}{r^{\beta}}\\ &  \le C b^d \eps_P^{-1/2} c_w^d \frac{r^{d}t}{r^{\beta}}  \big(C't/r^{\beta}\big)^{\alpha - 1} \\
& \le C'' \eps_P^{-\alpha/2} c_w^d \frac{r^{d}t^\alpha}{r^{\alpha\beta}}
\end{aligned}
   \] because $ \alpha \le 1 $ for $ \beta \ge 1.  $ For $ z \in \cZ_3\cup \cZ_4 $ \[ Q_X\big\{ 0 <\sqrt{\eps_P}\eta_z(x) \le 2t\big\} \le  C  c_w^d b^d  \frac{r^{d}t^\alpha}{r^{\alpha\beta}}  \]
Hence, \[ \begin{aligned}
Q_X\big\{ 0 <\left|\eta_Q^\sigma(x)-\frac12 \right| \le t\big\} & \le C \eps_P^{-\alpha/2} c_w^d \frac{r^{d}t^\alpha}{r^{\alpha\beta}} 2m + (r_1^{-d}-2m)C b^d c_w^d \frac{r^{d}t^\alpha}{r^{\alpha\beta}}\\
& \le C_\alpha t^\alpha.
\end{aligned} \] for small enough $ c_w. $

Let $  \frac{\sqrt{\eps_P}r^\beta}{6} \left(1-u_a(1/2)\right)^{1/\alpha}\le t \le \frac13. $ Then \[ \begin{aligned}
Q_X\left\{ 0 <\left|\eta_Q^\sigma(x)-\frac12 \right| \le t\right\} & \le 2m r_1^d + (r_1^{-d}-2m) (br_1)^d\\
& \le 2c_mc_w^d r^{\alpha\beta} + c_b^dr^d\\
& \le \big(2c_mc_w^d + c_b^d\big)r^{\alpha\beta}\\
& \le C_\alpha \left(\frac{\sqrt{\eps_P}r^\beta}{6} \left(1-u_a(1/2)\right)^{1/\alpha}\right)^{\alpha}\\
& \le C_\alpha t^\alpha
\end{aligned} \] where the second last inequality is true for small $ c_w $ and $ c_b. $

\paragraph{\textbf{Margin condition for parametric part}}

 Let $ t \le \frac12 \left(1-u_a(1/2)\right)^{1/\alpha}. $ 
%
For $ i \ge 1 $ \[ Q_X\big\{ 0 <\sqrt{\eps_P}\eta_{z_i}(x) \le 2t\big\} \le C  bc_w rt^\alpha. \]
  Hence, for $ t < c_Qn_Q^{-1/2} $ \[ \begin{aligned}
  Q_X\left\{ 0 <\left|\eta_Q^\sigma(x)-\frac12 \right| \le t\right\} & \le4D C  bc_w rt^\alpha\\
  & \le C_\alpha t^\alpha.
  \end{aligned} \] for small enough $ c_w. $ 
  For  $   c_Qn_Q^{-1/2} \le t \le \frac12 \left(1-u_a(1/2)\right)^{1/\alpha} $ \[ \begin{aligned}
  Q_X\left\{ 0 <\left|\eta_Q^\sigma(x)-\frac12 \right| \le t\right\} & \le4D C  bc_w rt^\alpha + c_w^\alpha r^{\alpha}\\
  & \le C_\alpha t^\alpha.
  \end{aligned} \] for small enough $ c_w. $

\end{proof}

\begin{lemma}
\label{lemma:kl-unlabeled-density}
Let $\Big\{\Pi_\sigma: \sigma \in \{-1, 1\}^{\cZ_1}\Big\}$ be the class of joint distributions defined in \ref{eq:dist-unlabeled-joint-density}. For $\sigma, \sigma' \in  \{-1, 1\}^{\cZ_1}$ we have \[ KL(\Pi_{\sigma}|| \Pi_{\sigma'})  \le c_w^dK(d, \alpha, \beta)\rho(\sigma, \sigma'). \]
\end{lemma}

\begin{proof}
We recall from \ref{eq:marginals-unlabeled-density} $Q_{X, \sigma}\sim \text{Uniform}\big([0,1]^d\big)$ doesn't depend on $\sigma.$ Hence, \[ KL\Big(Q_{X, \sigma}||Q_{X, \sigma'}\Big) = 0. \]

We refer to lemma C.4 in supplement   to show \[
\begin{aligned}
KL\Big(P_\sigma||P_{\sigma'}\Big) & = \frac 12 \Big(g_0^\sigma||g_0^{\sigma'}\Big) + \frac12 \Big(g_1^\sigma||g_1^{\sigma'}\Big)\\
& \le \eps_P c_w^d K'(d,\alpha, \beta ) r^{2\beta + d} \rho(\sigma, \sigma').
\end{aligned}
\]
Combining them we get \[ 
\begin{aligned}
KL\Big(\Pi_\sigma||\Pi_{\sigma'}\Big) & = n_P KL\Big(P_\sigma||P_{\sigma'}\Big) + n_Q KL\Big(Q_{X, \sigma}||Q_{X, \sigma'}\Big) \\
& \le n_P \eps_P  c_w^d K'(d,\alpha, \beta ) r^{2\beta + d} \rho(\sigma, \sigma')\\
& \le c_w ^d K'(d,\alpha, \beta ) c_r^{2\beta+d} (\eps_P n_P) (\eps_Pn_P)^{-1} \rho(\sigma, \sigma')\\
& \le c_w^d K(d, \alpha, \beta) \rho(\sigma, \sigma').
\end{aligned}
\]
\end{proof}

\begin{lemma}
	\label{lemma:verification-sepration-conditionals}
	For  suitable choices of $ c_w, b>0 $  	we have \[ \int_{\cX} (g_1^\sigma(x) - g_0^\sigma(x))^2 dx \ge C^2 .\]
\end{lemma}

\begin{proof}
	Notice that,  
	
	\[\xi_z(x) =  v \Bigg(\frac{2\|x-z\|_\infty }{br_1} - \frac2{r_1}\left(\frac1b-1\right)\Bigg) = 1 \] whenever $\|x-z\|_\infty\le r_1(1 - b)$ and in such regions the density differences are $\ge 1$. Hence, the integral of the difference squared is $\ge r_1^d (1-b)^d $ for non-parametric case and $\ge r_1 (1-b) $ for parametric case. Noticing that there are $(1/r_1^d - 2m)$ (non-parametric) and $2D$ (parametric)
many regions  we get the following	
	\[
\int_{\cX} \big(g_0(x) - g_1(x)\big)^2dx \ge \begin{cases}
(1 - 2c_m c_w^d)(1-b)^d \ \text{for non-parametric part},\\
\Big(1 - \frac1 {2D + 1}\Big)(1-b) \ \text{for parametric part.}
\end{cases}
\]  Here $\frac1{2D + 1} \sim c_w n_Q^{-\alpha/2}. $ The constants $ c_w, b$ can be suitably chosen to satisfy the condition in lemma.

	

\end{proof}

\section{Summary and discussion}
\label{sec:discussion}

We studied the hardness of the label shift problem in two settings, one in which the learner has access to labeled training examples from the target domain, and another in which the learner only has unlabeled training examples from the target domain. We showed that there is a difference between the hardness of the label shift problem in the two settings. In the former setting (in which the learner has access to labeled training examples from the target domain), the minimax rate is $O\big((\eps_Pn_P+n_Q)^{-{\beta}/(2\beta+d)} + n_Q^{-1/2} \big)^{1+\alpha}$, while in the latter setting, the minimax rate is $O\big( (\eps_Pn_P)^{-\beta/(2\beta+d)} + n_Q^{-1/2} \big)^{1+\alpha}$. We attribute this difference in rates is due to the availability of data from the target domain to estimate the the class conditional distributions in the former setting.  


Although we studied the hardness of the label shift problem with non-parametric model classes, we expect our results to generalize to more restrictive model classes. Inspecting the minimax rates, we see that they consist of two terms: a term that depends on the hardness of estimating the (ratio of) class conditional distributions, and a term that depends on the hardness of estimating the ratio of class probabilities in the source and target domains. In non-parametric classification, the hardness of estimating the class conditionals determines the hardness of the non-parametric classification problem in the IID setting \citep{kpotufe2017Lipschitz}. This observation leads us to interpret the first term in the minimax rate as the hardness of finding the optimal classifier, and we expect this term to change with the model class. Thus, for more restrictive model classes, we expect the first term in the minimax rates to improve (vanish faster) in a way that depends on the (reduced) complexity of the model class.

To wrap up, we mention two possible extensions of our work. First, it is natural to consider the label shift problem in high dimension. To keep the problem tractable, we must impose stronger parametric assumptions on the regression function. Such assumptions may also be phrased as assumptions on the class conditional densities because the regression function is (up to a monotone transform) the ratio of the class conditional densities. In the supervised label shift problem, we expect the minimax rate to depend on the hardness of estimating the regression function under the additional parametric assumptions. In the unsupervised label shift problem, we expect the distributional matching approach to provide good estimates of the class probability ratios (because the class probability ratios are low-dimensional), so we also expect the minimax rate to depend on the hardness of estimating the regression function. 

Second, it is natural to consider the possibility of achieving the minimax rate with a classifier that adapts to the smoothness of the regression function and the noise level in the labels. \citet{kpotufe2018Marginal} and \citet{cai2019Transfer} developed adaptive classifiers that attain the minimax rate in the covariate shift and posterior drift problems with Lepski's method. We expect Lepski's method will lead to an adaptive classifier in the label shift problem as well. However, the main goal of this paper is investigating the hardness of the label shift problem, and we defer such methodological questions to future work.

\bibliography{sm,YK}

\newpage
\appendix

\section{Simulation details}
\label{sec:simulation-details}
The codes and simple demonstrations are provided in https://github.com/smaityumich/label-shift. 
\paragraph{Data generating process}
We start by describing the data generating process $\cD(n, \pi)$. Let $ \mu_X $ denote the probability distribution of random variable $X$. For $a<b$ we define the $ \text{TN}(\mu, \sigma^2, a, b) $ as the $N(\mu,\sigma^2)$ distribution truncated to the interval $[a,b]$. 
Given inputs sample size $n$ and class probability $\pi$ for class $ 1 $, $\cD(n, \pi)$ returns a pair $(\bx,\by)$ where, 
\begin{itemize}
	\item $ \by $ is a $ n $ dimensional random vector with IID $\Ber(1,\pi)$ components.
	\item $ \bx = [x_1, \dots , x_{n}]^T $ is a $ n\times 3 $ random matrix with independent rows. The distribution of the $i$-th row is \[ x_i \mid y_i \sim y_i * \mu_{\text{TN}(0,1, -2, 2)}^{\otimes 3} + (1-y_i) * \mu_{\text{TN}\left(2,1, 0, 4\right)}^{\otimes 3}. \] 
	We observe that the features are supported on the hypercube $[-2,4]^{ 4}$


\end{itemize}

Given the data generating procedure $\cD(n, \pi)$ we generate the following synthetic data:
\begin{itemize}[label=\textendash]
	\item $ (\bx_P, \by_P) = \cD(n_P, 0.5) $ is the data from source population. 
	\item $ (\bx_Q, \by_Q) = \cD(n_Q, 0.75) $ is the data from target population.
	\item $ (\bx_{\text{test}}, \by_{\text{test}}) = \cD(n_{\text{test}}, 0.75) $ is the data for evaluating the performance of the classifiers, which shall also be referred as test data. Note the distribution of test data is same as the target distribution.
\end{itemize}

\paragraph{Other classifiers} Next we describe the classifiers that we shall consider for our comparative study: \begin{itemize}
	\item \texttt{Labeled-Classifier} is a function that takes the data $( \bx_P,  \by_P) $ from source, $ (\bx_Q,  \by_Q) $ from target distribution and  bandwidth parameters $ h_0, h_0>0 $ as inputs,  and returns the  classifier \[ \texttt{CL-labeled}  \triangleq  \texttt{Labeled-Classifier}(\bx_P, \by_P, \bx_Q, \by_Q,  h_0, h_1 )\] as defined in section 3, equation 3.1 and 3.3. Throughout our simulation study we use $\beta^*$-valid kernel (definition \ref{def:beta-valid-kernel}, \cite{tsybakov2009Introduction}, definition 1.2 and section 1.2.2) with $\beta^*$  as $ 3.$
	Since the densities are infinitely differentiable on the interior of support, we expect to realize a rate of convergence with $\beta = 3.$ In that regard, 
	we fix the bandwidth parameter $ h_0 = n_0 ^{-\frac 1{10}}, \   h_1 = n_1 ^{-\frac 1{10}} . $


			\item \texttt{Classical-Classifier} is a function that takes the target data $ (\bx_Q, \by_Q) $ and a bandwidth parameter $ h>0 $ as input and returns a classifier \[ \texttt{CL-classical} \triangleq \texttt{Classical-Classifier} (\bx_Q, \by_Q, h) \] where \[ \texttt{CL-classical}(x) =  \indicator\left\{  \frac{\sum_{i=1}^{n_Q} Y_i^{Q} K_{h}(x-X_i^Q)  }{\sum_{i=1}^{n_Q}  K_{h}(x-X_i^Q)} \ge \frac 12 \right\}. \] We fix $h_Q = \frac12 n_Q^{-\frac16}.$

			\item \texttt{Unlabeled-Classifier}   takes the data $ (\bx_P, \by_P) $ from source, $ x_Q $ from target distribution, a classifier $ g $ fitted on the source distribution and  bandwidth parameters $ h_0, h_1 >0 $ as inputs,  and returns the  classifier $$ \texttt{CL-unlabeled}  \triangleq  \texttt{Unlabeled-Classifier}(x_P, y_P, x_Q, g,   h )$$. It first estimates $\pi_Q$ using distribution matching approach \cite{lipton2018Detecting, azizzadenesheli2019Regularized, alexandari2020EM}. For this particular simulation study we use  \citet{lipton2018Detecting}.   The classifier $ g $ is a non-parametric classifier \cite{audibert2007fast} fitted on $x_P, \ y_P$ with same kernel and bandwidths. The final classifier is  obtained by an appropriate re-weighting of the $P$-samples.    We fix $h_0 = (n_0')^{-1/7}, h_1 = (n_1')^{-1/7}.$

		\item \texttt{Oracle-Classifier}  takes the source data $ x_P, \ y_P $, and  $ \pi_Q $ and  bandwidths $ h_0, h_1>0 $ as inputs,  and returns a classifier \[ \texttt{CL-oracle} \triangleq \texttt{Oracle-Classifier}(x_P, y_P, w_0, w_1, h) \] exactly same as in \texttt{Unlabeled-Classifier} with actual value  $\pi_Q$ used for data generating purpose.Here, we use same kernel and bandwidths.

\end{itemize}



\section{Proof of Theorem \ref{th:upper-bound} and \ref{th:upper-bound-unlabeled}}
\label{sec:ub-proofs}


\subsection{Definitions}

In this subsection we define $\beta$-valid kernel, convergence rates and parametric convergence rates.  

\begin{definition}[$ \beta $-valid kernel]
\label{def:beta-valid-kernel}
	Let $ K $ be a real-valued function on $ \reals^d, $ with support $ [-1, 1]^d. $ For fixed $ \beta>0, $ the function $ K(\cdot) $ is said to be a $ \beta $-valid kernel if it satisfies $ \int K=1, \ \int |K|^p < \infty $ for any $ p \ge 1, \ \int \|t\|^\beta |K(t)| dt < \infty  $ and, in the case $ \betafloor\ge 1,  $ it satisfies $ \int t^s K(t)dt = 0 $ for any $ s = (s_1, \dots , s_d) \in \bbN ^d$  such that  $1 \le |s| \le \betafloor.$
\end{definition}
We refer to \cite{tsybakov2009Introduction} section 1.2.2 for construction of  such kernels for 1 dimensional data. Kernel for $d $-dimensional data can be constructed as $K'(x_1, \dots, x_d) = K(x_1)\dots K(x_d). $




The proof of upper bound is broken into some technical lemmas. These lemmas are states in terms of general rate of convergence for parameter $\pi_Q$ and densities $g_0$ and $g_1$. Formal definitions of these rates are given later. 
 We denote the source-target sample size pair $(n_P, n_Q)$ by $n$, \ie\  $n \equiv (n_P, n_Q).$

\begin{definition}[Parameter estimation rate]
For $(P, Q)\in \Pi$    let  $\hat \theta_n$ is an estimator of the parameter $\theta = \theta(P, Q) \in \reals.$ For  non-increasing sequences $(\varphi_n)$ and $(\psi_n)$  of positive numbers we say $\hat \theta_n$ converges to $\theta$ at a $(\varphi_n, \psi_n)$-rate uniformly on $\cP$ if there exists positive numbers $c_1, c_2, c_\psi$  for any $\delta>0$ \[ \sup_{(P,Q)\in \cP} \Pr \big( |\hat \theta_n - \theta|>\delta \big) \le c_1 \exp\big(-c_2(\delta/\varphi_n)^2\big), \hspace{0.3cm} \delta \le c_\psi \psi_n. \] 
\end{definition}

\begin{definition}[Function estimation rate]
    For a pair $(P, Q) \in \cP$ let $\hat p_n$ be an estimator for $p \equiv p(P, Q) : \cX \to \reals. $ Let $(\varphi_n)$   be a sequence of non-increasing positive numbers. We say $\hat p_n $ converges pointwise to $p$ at a $\varphi_n$-rate for $(P, Q)$ if there exists positive constants $c_1, c_2,\Delta$ and $c_\varphi$ such that  for $Q_X$ almost surely all $x\in \cX$ we have \[  \Pr \big( |\hat p_n(x) -  p(x) |>\delta \big) \le c_1\exp\big(- c_2 (\delta/\varphi_n)^2\big), \hspace{0.5cm} c_\varphi\varphi_n < \delta <\Delta. \] We say $\hat p_n $ converges pointwise to $p$ at a $\varphi_n$-rate uniformly on $\cP,$ if the above happens for all $(P, Q)\in \cP$ for some constants $c_1, c_2, \Delta$ and $c_\varphi$ independent of $(P, Q).$
\end{definition}



\subsection{Required Lemmas}
The proof of upper bound is broken into three main lemmas, which are presented in this subsection. 
\begin{lemma}[Concentration of $ \hat\eta_Q $]
\label{lemma:conc-reg-fn}
Let $\hat\pi_Q^{(n)}$ converges to $\pi_Q$ at a $(\varphi_n, \psi_n)$-rate and for $i \in \{0, 1\}$ let $\hat g_i^{(n)}$ converges pointwise to $g_i$ at a $\tau_n^{(i)}$ rate, uniformly on $\cP.$ Then there exists positive constants $c_0, c_1, c_2$ such that	$\hat \eta_Q$ converges pointwise to $\eta_Q$ at a $\big( c_0(1-\pi_Q + \psi_n ) \tau_n^{(0)}+ c_1 (\pi_Q + \psi_n ) \tau_n^{(1)} + c_2 \varphi_n \big)$-rate uniformly over $\cP. $
\end{lemma}
\begin{proof} We break the proof in several steps.
	
\noindent	\underline{\textit{Step 1: Upper bound for $ \hat \eta_Q. $}}
	Note that both $ \hat \eta_Q(x) $ and $ \eta_Q(x) $ can be expresses as 
	\begin{align*}
		\eta_Q(x) & = \frac{\pi_Q g_1(x)}{\pi_Q g_1(x) + (1-\pi_Q) g_0(x)} \\
		 \hat \eta_Q(x) & = \frac{\hat \pi_Q \hat g_1(x)}{\hat \pi_Q \hat g_1(x) + (1-\hat \pi_Q)\hat g_0(x)}.
	\end{align*} For the ease of notation, let us define $ u(x) = \pi_Qg_1(x), \ v(x) = (1-\pi_Q)g_0(x), \ \hat u(x) = \hat \pi_Q\hat g_1(x) $ and $ \hat v(x) = (1-\hat \pi_Q) \hat g_0(x). $ Then 
	\begin{align*}
		|\hat \eta_Q(x) - \eta_Q(x)| = & \left| \frac{\hat u(x)}{\hat u(x) + \hat v(x)} - \frac{u(x)}{u(x) + v(x)} \right|\\
		= & \frac{|\hat u(x)v(x) - u(x)\hat v(x)|}{(\hat u(x)+\hat v(x))(u(x)+v(x))}\\
		= & \frac{|\hat u(x)v(x) - \hat u(x)\hat v(x) + \hat u(x)\hat v(x) - u(x)\hat v(x)|}{(\hat u(x)+\hat v(x))(u(x)+v(x))}\\
			= & \frac{|\hat u(x)[v(x) -\hat v(x)] +\hat v(x)[ \hat u(x) - u(x)]|}{(\hat u(x)+\hat v(x))(u(x)+v(x))}\\
		\le & \frac{\hat u(x)|v(x) -\hat v(x)| +\hat v(x)| \hat u(x) - u(x)|}{(\hat u(x)+\hat v(x))(u(x)+v(x))}\\
		\le & \frac{|\hat u(x) - u(x)|+ |\hat v(x) - v(x)|}{q_X(x)}.
	\end{align*}
	
\noindent	\underline{\textit{Step 2: Upper bound of $ \hat u (x) $ and $ \hat v(x) .$}}
	Since,  $\hat \pi_Q^{(n)}\equiv \hat \pi_Q$ converges to $\pi_Q$ at a $(\varphi_n, \psi_n)$-rate, there exists $c_1^{(\pi)}, c_2^{(\pi)}, c_\pi>0$ such that for any $\delta>0$ \[  \sup_{(P, Q)\in \Pi} \Pr \Big( \big| \hat \pi_Q^{(n)} - \pi_Q \big|>\delta \Big)\le c_1^{(\pi)}\exp \Big(-c_2^{(\pi)}(\delta/\varphi_n)^2\Big), \hspace{0.3cm} \delta \le c_\pi\psi_n \]
	The above inequality can be rewritten as \begin{equation}
	    \sup_{(P, Q)\in \Pi} \Pr \Big( \big| \hat \pi_Q^{(n)} - \pi_Q \big|>\delta \varphi_n/\sqrt{c_2^{(\pi)}} \Big)\le c_1^{(\pi)}\exp \big(-\delta^2\big), \hspace{0.3cm} \delta \le c_\pi\psi_n \sqrt{c_2^{(\pi)}}/\varphi_n.
	\end{equation}
	Fix $i\in \{0, 1\}.$
	Since,  $\hat g_i^{(n)}$ converges pointwise to $g_i$ at a $\tau_{n}^{(i)}$-rate, there exists positive constants $c_{1,i}, c_{2,i}, \Delta_i, c_{\tau, i}$ such that for any $(P, Q)\in \cP$ for $Q_X$ almost surely on $\cX$ we have 
	\[  \Pr \Big(\big|\hat g_i^{(n)}(x)-g_i(x)\big|>\delta\Big) \le c_{1,i}\exp\big(-c_{2,i}(\delta/\tau_{n}^{(i)})^2\big)\hspace{0.5in} c_{\tau, i} \tau_n^{(i)} < \delta < \Delta_i. \]
	The above inequality can be rewritten as 
	\begin{equation}
	    \Pr \Big(\big|\hat g_i^{(n)}(x)-g_i(x)\big|>\delta\tau_n^{(i)}/\sqrt{c_{2,i}}\Big) \le c_{1,i}\exp\big(-\delta^2\big),\hspace{0.3in} c_{\tau, i}  < \delta < \Delta_i/\tau_n^{(i)}.
	\end{equation}
	
	Using union bound, for $Q_X$ almost surely $x\in \cX$, with probability at least $1- (c_{1,0}+c_{1,1}+c_1^{(\pi)})e^{-\delta^2}= 1- c_1'e^{-\delta^2}$ we have the following
	\begin{align}
	    \big| \hat \pi_Q^{(n)} - \pi_Q \big|\le &\delta \varphi_n/\sqrt{c_2^{(\pi)}}  \\
	    \big|\hat g_0^{(n)}(x)-g_0(x)\big|\le & \delta\tau_n^{(0)}/\sqrt{c_{2,0}}  \\
	     \big|\hat g_1^{(n)}(x)-g_1(x)\big|\le & \delta\tau_n^{(1)}/\sqrt{c_{2,1}}  
	\end{align}
	for $c_{\tau, 0} \vee c_{\tau, 1} < \delta < (\Delta_0/\tau_n^{(0)}) \wedge (\Delta_1/\tau_n^{(1)})\wedge (c_\pi\psi_n\sqrt{c_2^{(\pi)}}/\varphi_n).$
	From the above inequalities, we get
	\begin{align*}
	 	|\hat u(x) - u(x)| & = |\hat \pi_Q^{(n)}\hat g_1^{(n)}(x) - \pi_Qg_1(x)|\\
		& = |\hat \pi_Q\upn \hat g_1\upn(x) - \hat \pi_Q\upn g_1(x) + \hat \pi_Q\upn g_1(x) - \pi_Q g_1(x)|\\
		& \le \hat \pi_Q\upn  |\hat g_1\upn (x) - g_1(x)| + g_1(x) |\hat \pi_Q\upn  - \pi_Q| \\
		& \le \big(|\hat \pi_Q\upn - \pi_Q| + \pi_Q\big) |\hat g_1\upn (x) - g_1(x)| + g_1(x) |\hat \pi_Q\upn  - \pi_Q|\\
		& \le \big(\pi_Q + \delta \varphi_n/\sqrt{c_2^{(\pi)}} \big)
		\delta\tau_n^{(1)}/\sqrt{c_{2,1}} + L^* \delta \varphi_n/\sqrt{c_2^{(\pi)}}\\
		& \le \big(\pi_Q + c_\pi \psi_n \big)
		\delta\tau_n^{(1)}/\sqrt{c_{2,1}} + L^* \delta \varphi_n/\sqrt{c_2^{(\pi)}},
	\end{align*} and similarly \[ |\hat v(x) - v(x)|\le \big(1- \pi_Q + c_\pi \psi_n \big)
		\delta\tau_n^{(0)}/\sqrt{c_{2,0}} + L^* \delta \varphi_n/\sqrt{c_2^{(\pi)}} . \] 
	
\noindent	\underline{\textit{Step 3: Concentration of $ \eta_Q $.}}
	Under strong density assumption, for $Q_X$ almost surely $x\in \cX$, with probability at least $ 1- c_1'e^{-\delta^2}$ we have
	
		\begin{align*}
		     |\hat \eta_Q(x) - \eta_Q(x)|  \le & \frac{\delta}{\mu_+}\Bigg[\big(\pi_Q + c_\pi \psi_n \big)
		\tau_n^{(1)}/\sqrt{c_{2,1}}\\ & + \big(1- \pi_Q + c_\pi \psi_n \big)
		\tau_n^{(0)}/\sqrt{c_{2,0}} +2 L^*  \varphi_n/\sqrt{c_2^{(\pi)}} \Bigg]\\
		= & \delta r_n 
		\end{align*}
		for $c_{\tau, 0} \vee c_{\tau, 1} < \delta < (\Delta_0/\tau_n^{(0)}) \wedge (\Delta_1/\tau_n^{(1)})\wedge (c_\pi\psi_n\sqrt{c_2^{(\pi)}}/\varphi_n).$
	We can rewrite this as for $Q_X$ almost surely $x\in \cX$ 
	\[
	\sup_{(P, Q)\in \cP} \Pr \Big( |\hat \eta_Q(x) - \eta_Q(x)| > \delta  \Big) \le c_1' \exp \big(- (\delta / r_n)^2\big)
	\]  for $\big(c_{\tau, 0} \vee c_{\tau, 1}\big)r_n < \delta < r_n\Big[(\Delta_0/\tau_n^{(0)}) \wedge (\Delta_1/\tau_n^{(1)})\wedge (c_\pi\psi_n\sqrt{c_2^{(\pi)}}/\varphi_n)\Big].$ Let  $ c_r = c_{\tau, 0} \vee c_{\tau, 1}.$ Since, $r_n/\tau_n^{(0)} \ge 1/ \sqrt{c_{2, 0}}$, $r_n/\tau_n^{(1)} \ge 1/ \sqrt{c_{2, 1}}$ and $ r_n \psi_n/\varphi_n \ge c_\pi \sqrt{c_2^{(\pi)}} $ letting $\Delta = (\Delta_0 / \sqrt{c_{2, 0}}) \wedge (\Delta_1 / \sqrt{c_{2, 1}}) \wedge (c_\pi \sqrt{c_2^{(\pi)}})$ we get, for $Q_X$ almost surely $x\in \cX$ 
	\[
	  \sup_{(P, Q)\in \cP} \Pr \Big( |\hat \eta_Q(x) - \eta_Q(x)| > \delta  \Big) \le c_1' \exp \big(- (\delta / r_n)^2\big), \hspace{0.3cm} c_r r_n < \delta < \Delta.  
\]

\end{proof}

\begin{lemma}[Bound on $ \Ex\cE_Q(\hat f)  $]
\label{lemma:margin-bound}
	Suppose an estimate $\hat \eta_Q $ of the regression function $ \eta_Q $ converges pointwise at a $r_n$-rate uniformly on $\cP$. Then under $\alpha$-margin condition  there exists a positive constant $C$ such that 
	 \[ \sup_{(P, Q) \in \cP} \Ex\Big[\cE_Q(\hat f)\Big] \le C r_n^{{1+\alpha}}.  \] .
\end{lemma}
\begin{proof}
Since, $\hat \eta_Q$ converges pointwise to $\eta_Q$ at a $r_n$-rate, there exist positive constants $c_1, c_2, \Delta, c_r$ such that for $Q_X$ almost surely all $x\in \cX$ \[
\sup_{(P, Q)\in \cP} \Pr \Big( |\hat \eta_Q(x) - \eta_Q(x)|>\delta \Big) \le c_1 \exp\big(- c_2 (\delta/r_n)^2 \big), \hspace{0.3cm} c_rr_n <\delta < \Delta.
\] 
Recall, under $\alpha$-margin condition there exists $c_\alpha>0$ such that  \[
Q_X \big(|\eta_Q(x) - 1/2|>\delta \big) \le c_\alpha \delta ^\alpha.
\]
We replace $c_\alpha$ by $c_\alpha (\Delta/2)^{-\alpha} \vee 1$ so that $c_\alpha(\Delta/2)^\alpha\ge 1. $

		We define the following events. $$ A_0 = \left\{ x\in \reals^d: 0< \left| \eta_Q(x) - 1/2\right| < \delta         \right\} $$ and for $j \ge 1, $ $$   A_j = \left\{ x\in \reals^d: 2^{(j-1)}\delta< \left| \eta_Q(x) - 1/2\right| < 2^j\delta       \right\}  $$ Now, \begin{align*}
	\cE_Q(\hat f)  = & 2 \Ex_X \left(\left| \eta_Q(X) -1/2 \right|\indicator_{\{\hat f(X)\neq f^*(X)\}}  \right) \\
	= & 2\sum_{j=0}^\infty \Ex_X\left(\left| \eta_Q(X) -1/2 \right|\indicator_{\{\hat f(X)\neq f^*(X)\}} \indicator_{\{X\in A_j\}} \right) \\
	\le & 2 \delta \Ex_X\left( 0< \left| \eta_Q(X) -\frac12 \right|< \delta  \right) \\
	& + 2\sum_{j=1}^\infty \Pr_X\left(\left| \eta_Q(X) -\frac12 \right|\indicator_{\{\hat f(X)\neq f^*(X)\}} \indicator_{\{X\in A_j\}} \right)
	\end{align*}
	
	Let $\delta = c_r r_n.$ Then $2^{j-1}\delta \ge \Delta/2  $ if $j \ge  \log_2 (\Delta / \delta).$
	
	On the event $ \{ \hat f\neq f^*  \} $ we have $ \left| \eta_Q - \frac12 \right| \le \left| \hat \eta - \eta \right|. $ So, for any $1\le  j < \log_2 (\Delta / \delta)$  we get \begin{align*}
	2\Ex_X&\Ex \left(\left| \eta_Q(X) -\frac12 \right|\indicator_{\{\hat f(X)\neq f^*(X)\}} \indicator_{\{X\in A_j\}} \right) \\
	& \le 2^{j+1}\delta \Ex_X\Ex \left(\indicator_{\{|\hat \eta_Q(X)- \eta_Q(X)| \ge2^{j-1}\delta \}} \indicator_{\{0< |\eta_Q(X)-1/2|<2^{j}\delta\}} \right) \\
	& = 2^{j+1}\delta \Ex_X \left[\Pr\left(\indicator_{\{|\hat \eta_Q(X)- \eta_Q(X)| \ge2^{j-1}\delta \}}\right) \indicator_{\{0< |\eta_Q(X)-1/2|<2^{j}\delta\}} \right] \\
	& \le 2^{j+1}\delta \exp\left(-(2^{j-1}\delta/r_n)^2\right)\Pr_X (0< |\eta_Q(X)-1/2|<2^{j}\delta)\\
	& \le 2C_\alpha 2^{j(1+\alpha)}\delta ^{1+\alpha} \exp\left(-(2^{j-1}\delta/r_n)^2\right).
	\end{align*}
	
	For $j \ge \log_2 (\Delta / \delta)$ we have  $\Pr \big( |\eta _Q(x) - 1/2|\ge 2^{j-1}\delta \big) = 1$ and hence \[ \Pr \Big( 2^{j-1}\delta <|\eta_Q(x) - 1/2|< 2^{j}\delta \Big ) = 0. \] This means \begin{align*}
	2\Ex_X&\Ex \left(\left| \eta_Q(X) -\frac12 \right|\indicator_{\{\hat f(X)\neq f^*(X)\}} \indicator_{\{X\in A_j\}} \right) \\
	& \le 2^{j+1}\delta \Ex_X\Ex \left(\indicator_{\{|\hat \eta_Q(X)- \eta_Q(X)| \ge2^{j-1}\delta \}} \indicator_{\{2^{j-1}\delta< |\eta_Q(X)-1/2|<2^{j}\delta\}} \right) \\
	& = 2^{j+1}\delta \Ex_X \left[\Pr\left(\indicator_{\{|\hat \eta_Q(X)- \eta_Q(X)| \ge2^{j-1}\delta \}}\right) \indicator_{\{2^{j-1}\delta< |\eta_Q(X)-1/2|<2^{j}\delta\}} \right] \\
	& = 0.
	\end{align*}
	Finally,  we get \begin{align*}
	\sup_{\Pr\in \cP} \Ex\Big[\cE_Q(\hat f)\Big] & \le 2C_\alpha \left( \delta^{1+\alpha} + \sum_{j\ge 1}  2^{j(1+\alpha)} \delta^{1+\alpha}\exp\left(-(2^{j-1}\delta/r_n)^2\right)   \right) \\
	& \le C r_n^{{1+\alpha}}.
	\end{align*}
\end{proof}
Now we provide a rate of convergence for the density estimator $\hat g_1^{(n_1)}$. Later we only provide the statement for the rate of convergence for the density estimator $\hat g_0 ^{(n_1)}.$ The proof will be similar. 
\begin{lemma}[Rate of convergence for conditional density estimates]
\label{lemma:rate-g1}
 Let $(P, Q) \in \cP.$ For $i = 0, 1$ let $m_i(n) \equiv m_i(n, \pi) = \Ex [n_i].$ Then for $h_i = n_i^{-1/(2\beta + d)}$ the density estimator $\hat g_i^{(n_i)}$ converges pointwise to   $g_i$ at a $m_i(n)^{-\beta/(2\beta + d)}$-rate.
\end{lemma}

\begin{proof}
We only prove the result for $g_1.$ The proof for $g_0$ will be similar. 
         Let $n_1$  be the number of sample points with label $1$ (which is sum of independent Bernoulli variables). Let $v_1(n) = \text{Var}(n_1).$ Clearly, $v_1(n) \le m_1(n).$
         
         Using Bernsteins's inequality, for any $t>0$ we get 
         \begin{align*}
             \Pr \left( |n_1 - m_1(n)|>t \right)
            \le & 2 \exp \left( - \frac{t^2/2}{v_1(n) + t/3} \right) \\
            \le & 2 \exp \left( - \frac{t^2/2}{2v_1(n)} \right) \text{ for } t \le 3v_1(n),\\
            \le & 2 \exp \left( - \frac{t^2}{4m_1(n)} \right) \text{ for } t \le 3v_1(n).
         \end{align*}
         Letting $t = \delta m_1(n)^{(2\beta+d/2)/(2\beta + d) }$ we get \[ \Pr \left( |n_1 - m_1(n)|>\delta m_1(n)^{\frac{2\beta+d/2}{2\beta + d} }\right) \le  2 \exp \left( - \frac{\delta^{2}}4m_1(n)^{\frac{2\beta}{2\beta + d} } \right),\] for $ \delta \le \frac{3v_1(n)}{m_1(n)^\frac{2\beta+d/2}{2\beta + d}}.$  
         
         Using \cite[Lemma 4.1]{rigollet2009optimal}  we get positive constants $c_1, c_2, c', \Delta$ such that for $Q_X$ almost sure all $x \in \cX$ \[
         \sup_{(P, Q) \in \cP} \Pr \Big( \big|\hat g_1^{(n_1)}(x) - g_1(x) \big| > \delta \Big| n_1 \Big) \le c_1 \exp\big( - c_2 n_1 h_1^d \delta^2 \big), \hspace{0.3cm} c'h_1^\beta < \delta < \Delta.
         \] 
         Letting $h_1 = n_1^{-\frac1{2\beta + d}}$ we get 
         \[
         \sup_{(P, Q) \in \cP} \Pr \Big( \big|\hat g_1^{(n_1)}(x) - g_1(x) \big| > \delta \Big| n_1 \Big) \le c_1 \exp\big( - c_2 n_1 ^{\frac{2\beta}{2\beta + d}} \delta^2 \big), \hspace{0.3cm} c'n_1^{-\frac\beta{2\beta + d}} < \delta < \Delta.
         \] 
         
         Let $n^{(0)} = (n_P^{(0)}, n_Q^{(0)})$ such that $\Delta \le \left(\frac{3v_1(n^{(0)})}{m_1(n^{(0)})^\frac{(2\beta+d/2)}{(2\beta + d)}}\right) \bigwedge\left( \frac12 m_1(n^{(0)})^\frac{d/2}{2\beta + d}\right).$ We say $n\ge n^{(0)}$ if $n_P  \ge n_P^{(0)}$ and $n_Q \ge n_Q ^{(0)}.$ For any $n \ge n^{(0)}$ we have  $\delta m(n)^{\frac{2\beta+d/2}{2\beta + d} } \le m(n)/2. $ 
         
         Now, for $n \ge n^{(0)}$, and  $Q_X$ almost surely all $x \in \cX$  \begin{align*}
            &  \Pr \Big( \big|\hat g_1^{(n_1)}(x) - g_1(x) \big| > \delta  \Big) \\ \le &  \Pr \Big( \big|\hat g_1^{(n_1)}(x) - g_1(x) \big| >  \delta, |n_1 - m_1(n)|\le \delta m_1(n)^{\frac{2\beta+d/2}{2\beta + d} }  \Big) \\
              & + \Pr \left( |n_1 - m_1(n)|>\delta m_1(n)^{\frac{2\beta+d/2}{2\beta + d} }\right)  \\
              \le & \Pr \Big( \big|\hat g_1^{(n_1)}(x) - g_1(x) \big| >  \delta, |n_1 - m_1(n)|< m_1(n)/2  \Big) \\
              & + 2 \exp \left( - \frac{\delta^{2}}4m_1(n)^{\frac{2\beta}{2\beta + d} } \right)\\
              \le & \Ex \left[ \Pr \Big( \big|\hat g_1^{(n_1)}(x) - g_1(x) \big| > \delta \Big| n_1 \Big) \indicator_{m_1(n)/2 \le n_1 \le 3m_1(n)/2}  \right]\\
               & + 2 \exp \left( - \frac{\delta^{2}}4m_1(n)^{\frac{2\beta}{2\beta + d} } \right)\\
               \le & c_1 \exp\left( - c_2 \left(\frac{m_1(n)}2\right) ^{\frac{2\beta}{2\beta + d}} \delta^2 \right)\\
                & + 2 \exp \left( - \frac{\delta^{2}}4m_1(n)^{\frac{2\beta}{2\beta + d} } \right),  \text{for } c'\left( \frac{m_1(n)}{2} \right)^{-\frac\beta{2\beta + d}} < \delta < \Delta.
         \end{align*}
         Hence, there exists $c_1', c_2', c', \Delta$ such that for $Q_X$ almost surely all $x \in \cX$ \[
          \Pr \left( |\hat g_1^{(n_1)}(x) - g_1(x)| > \delta \right) \le c_1'\exp\left(-c_2' m_1(n)^{\frac{2\beta}{2\beta + d}} \delta^2 \right),
         \] for $ c'm_1(n)^{-\frac\beta{2\beta + d}} < \delta < \Delta.$
\end{proof}

\subsection{Upper Bounds}


\subsubsection{Proof of Theorem \ref{th:upper-bound}}
\label{proof:upper-bound-labeled}
\begin{proof}
    Since, $\hat \pi_Q^{(n)} = \frac 1{n_Q} \sum_{i=1}^{n_Q} Y_i^{(Q)}$     using Bernstein's inequality 
    \[ 
    \Pr \left( \big|\hat \pi_Q^{(n)}- \pi_Q\big|>t \right) \le 2\exp\left( - \frac{t^2/2}{\pi_Q(1-\pi_Q)/n_Q+ t/(3n_Q)} \right).
    \] Letting $t \le 3\pi_Q(1-\pi_Q)$ we have 
    \[ 
    \Pr \left( \big|\hat \pi_Q^{(n)}- \pi_Q\big|>t \right) \le 2\exp\left( - \frac{n_Qt^2}{4\pi_Q(1-\pi_Q)} \right) .
    \]
    Hence, $\hat \pi_Q ^{(n)}$ converges to $\pi_Q$ at a $((\pi_Q(1-\pi_Q))/n_Q)^{1/2}, \pi_Q(1-\pi_Q))$-rate uniformly over $\cP$.

    Since $n_1 = \sum_{i=1}^{n_P} Y_i^{(P)} + \sum_{i=1}^{n_Q} Y_i^{(Q)}$ from lemma \ref{lemma:rate-g1} we see that $\hat g_1^{(n_1)}$ converges pointwise to   $g_1$ at a $m(n)^{-\beta/(2\beta + d)}$-rate uniformly on $\cP$. Similarly, $\hat g_0^{(n_1)}$ converges pointwise to   $g_0$ at a $\big(n_P + n_Q - m(n)\big)^{-\beta/(2\beta + d)}$-rate.  We refer to lemma \ref{lemma:conc-reg-fn} to conclude 
    $\hat \eta_Q$ converges to $\eta_Q$ at a rate  
   \begin{equation}
   \label{eq:reg-fn-rate-labeled}
     r(n, \pi)= \left\{ \begin{aligned}
    & \sqrt{\frac{(1-\pi_Q)\pi_Q}{n_Q}} +  \sqrt{\pi_Q} \Big( \pi_P n_P + \pi_Q n_Q \Big)^{- \frac{\beta}{2\beta + d}} \\ & +  \sqrt{ (1-\pi_Q)} \Big( (1-\pi_P) n_P + (1-\pi_Q) n_Q \Big)^{- \frac{\beta}{2\beta + d}} 
    \end{aligned}  \right.  = (A) + (B) + (C)
   \end{equation} uniformly on $\cP. $ 
   
   To complete the proof, we derive the worst possible rate over the function class. First, we seek the worst rate with respect to $\pi_Q$.
    Considering them term by term we see that $(A)\le \frac1{2\sqrt{n_Q}}, $ $(B)$ is an increasing function of $\pi_Q$ and hence \[(B) \le \big( \pi_P n_P +  n_Q \big)^{-\frac{\beta}{2\beta + d}}\,.\] By similar logic we see that 
\[(C) \le \big( (1-\pi_P) n_P +  n_Q \big)^{-\frac{\beta}{2\beta + d}}.\] Hence, 
\[ 
\begin{aligned}
r(n, \pi) &\le \frac1{2\sqrt{n_Q}} + \big( \pi_P n_P +  n_Q \big)^{-\frac{\beta}{2\beta + d}} + \big( (1-\pi_P) n_P +  n_Q \big)^{-\frac{\beta}{2\beta + d}}\\
& \le \frac1{\sqrt{n_Q}} + \big( \pi_P n_P +  n_Q \big)^{-\frac{\beta}{2\beta + d}} + \big( (1-\pi_P) n_P +  n_Q \big)^{-\frac{\beta}{2\beta + d}}
\end{aligned}
\,.\] But a concern regarding the above calculation is whether this dominating rate is achievable?
The following argument shows that it is  achieved by $\pi_Q = 1/2.$

\[
\begin{aligned}
r(n, \pi; \pi_Q = 1/2) &= \frac1{2\sqrt{n_Q}} + \frac1{\sqrt{2}} \big( \pi_P n_P +  n_Q/2 \big)^{-\frac{\beta}{2\beta + d}}\\
& \quad + \frac1{\sqrt{2}} \big( (1-\pi_P) n_P +  n_Q/2 \big)^{-\frac{\beta}{2\beta + d}}\\
& \ge \frac1{4\sqrt{n_Q}} + \frac1{{4}} \big( \pi_P n_P +  n_Q \big)^{-\frac{\beta}{2\beta + d}}\\
& \quad+ \frac1{{4}} \big( (1-\pi_P) n_P +  n_Q \big)^{-\frac{\beta}{2\beta + d}}\\
& = \frac14 \left[\frac1{\sqrt{n_Q}} + \big( \pi_P n_P +  n_Q \big)^{-\frac{\beta}{2\beta + d}} \right.\\
&\quad \left.+ \big( (1-\pi_P) n_P +  n_Q \big)^{-\frac{\beta}{2\beta + d}}\right]
\end{aligned}
\]
This explains why $\pi_Q = 1/2$ exhibits the worst behavior.

Next, denoting $f(\pi_P) = \big( \pi_P n_P +  n_Q \big)^{-\frac{\beta}{2\beta + d}}$ we note that 
\[
\max \Big\{f(\pi_P), f(1-\pi_P) \Big\} \le f(\pi_P)+ f(1-\pi_P) \le 2\max \Big\{f(\pi_P), f(1-\pi_P) \Big\},
\] which implies $\max \Big\{f(\pi_P), f(1-\pi_P) \Big\}$ and $f(\pi_P)+ f(1-\pi_P)$ have same rate of convergence. Furthermore, using the fact that $f$ is a decreasing function we get 
\[
\max \Big\{f(\pi_P), f(1-\pi_P) \Big\} = f(\min\{\pi_P, 1-\pi_P\}) \le f(\eps_P),
\] where the last inequality is an equality in the worst possible case ($\pi_P = \eps_P$ or $1-\eps_P$). Hence, we finally get the worst possible rate when $\pi_Q = 1/2$ and $\pi_P = \epsilon_P$ or $1 - \epsilon_P$: 
\[ n_Q^{-1/2} + \big( \eps_P n_P +  n_Q \big)^{-\frac{\beta}{2\beta + d}}. \]



     
     Finally, under $\alpha$-margin condition lemma \ref{lemma:margin-bound} we have 
    \begin{align*}
        \sup_{(P, Q)\in \cP} \Ex \cE_Q(\hat f) & \le C\left( n_Q^{-1/2} +  (\eps_P n_P+n_Q)^{-\frac{\beta}{2\beta + d}}\right)^{1+\alpha}.
    \end{align*}
\end{proof}


\subsubsection{Proof of Theorem \ref{th:upper-bound-unlabeled}}

\begin{proof}
	Let $n_1 = \sum_{i=1}^{n_P} Y_i^{(P)}$ and $n_0 = n_P  - n_1$.

	We use the method by \cite{iyer2014Maximum} to estimate $\pi_Q$ with a Gaussian kernel $K_c(x, y) = A \exp(-c\|x-y\|_2^2)$, where $A>0$ is suitably chosen to satisfy $\int_{\cX^2}K_c(x, y)dxdy = 1.$ This ensures $K_c$ to be a joint density on $\cX^2$. Furthermore, at the limit $c\to\infty$ the joint probability distribution corresponding to $K_c$ converges to an uniform distribution on the line $x = y$. Since $\big(g_0(x) - g_1(x)\big)\big(g_0(y) - g_1(y)\big)$ is a continuous function on a compact support $\cX^2$, this is bounded. We apply bounded convergence theorem to conclude 
\[\lim_{c\to \infty} \int_{\cX^2}\big(g_0(x) - g_1(x)\big)\big(g_0(y) - g_1(y)\big) K_c(x, y) dxdy = \int_{\cX} \big(g_0(x) - g_1(x)\big)^2dx. \]
From the Assumption \ref{assump:separation-conditionals} we have \[
\int_{\cX} \big(g_0(x) - g_1(x)\big)^2dx \ge C^2\,,
\]
which implies 
\[\int_{\cX^2}\big(g_0(x) - g_1(x)\big)\big(g_0(y) - g_1(y)\big) K_c(x, y) dxdy \ge \frac {C^2}2,\] for a large enough $c>0$. 
Denoting  the corresponding feature vector (of the kernel $K_c$) as $\Phi_c$ we see that $\bar A$ in \cite{iyer2014Maximum}, Equation (2) is merely $\int \Phi_c(x) (g_1(x) - g_0(x)) dx.$
Hence we have \[
\begin{aligned}
\bar A^\top \bar A & = \left\langle \int \Phi_c(x) (g_1(x) - g_0(x)) dx, \int \Phi_c(y) (g_1(y) - g_0(y)) dy \right\rangle\\
& = \int _{\cX^2} \left\langle \Phi_c(x), \Phi_c(y) \right\rangle \big(g_0(x) - g_1(x)\big)\big(g_0(y) - g_1(y)\big) dxdy\\
&= \int_{\cX^2} K_c(x, y)\big(g_0(x) - g_1(x)\big)\big(g_0(y) - g_1(y)\big)  dxdy \ge  \frac {C^2}2
\end{aligned}
\]

Using the statement right after Lemma 2 in \cite{iyer2014Maximum} we get 

\[ 2(\hat \pi_Q - \pi_Q)^2 \le \frac{R^2\left(\frac{c^2 + 2c + 2}{n_Q} + \frac 2{n_0} + \frac 2{n_1}\right)\left(1 + \sqrt{\log(4/\delta)}\right)^2 }{{\bar A}^\top {\bar A}- 8R^2\left(\frac1{n_0} + \frac1{n_1}\right) \sqrt{\left(\frac {c^2}{n_0} + \frac1{n_1}\right) \log(2/\delta)}} \]
with probability at least $1-\delta$. Here $c = 1$, $R = \max_{x\in \cX} \|\Phi_c\|$, $n_0$ is the number of observations with $Y = 0$ in source and similarly $n_1$. We note that the numbers $R$ and ${\bar A}^\top {\bar A} \ge \frac {C_k^2}2$ are fixed and $n_0, n_1, n_Q$ are growing. Hence, for sufficiently large $n_0, n_1$ and $n_Q$ we get 
\[
|\hat \pi_Q - \pi_Q| \le C' \left(\sqrt{\frac{1}{n_Q} + \frac 1{n_0} + \frac 1{n_1}}\right) \left(1 + \sqrt{\log(4/\delta)}\right)
\] with probability at least $1-\delta$. In other words, there is a $c>0$ such that  for any $ n_1 $ with probability (conditioned over $ n_1 $) $ \ge 1- e^{-t^2} $ the following holds \[ |\hat \pi_Q- \pi_Q| \le ct \sqrt{\frac1{n_Q} + \frac1{n_0}+\frac1{n_1}}\,.  \]  Recall from the Bernstein inequality in \ref{lemma:rate-g1} we have \[ \pi_P n_P - c't\sqrt{n_P \pi_P (1-\pi_P)} \le  n_1 \le \pi_P n_P +c't \sqrt{n_P \pi_P (1-\pi_P)} \] with probability $\ge 1 - e^{-t^2} $ for $ t \le 2\sqrt{n_P\pi_P(1-\pi_P)} $ for some $ c' .$ Hence, for $ t \le 1\sqrt{2c'n_P\pi_P(1-\pi_P)} $ with probability $ \ge 1- 2e^{-t^2} $ we have 
	\[ \begin{aligned}
	 & |\hat \pi_Q- \pi_Q|\\
	  & \le ct \sqrt{\frac1{n_Q} + \frac1{n_0}+\frac1{n_1}}\\
	 & \le ct \sqrt{\frac1{n_Q} + \frac1{(1-\pi_P)n_P - c't \sqrt{n_P \pi_P (1-\pi_P)} }+\frac1{\pi_Pn_P - c't\sqrt{n_P \pi_P (1-\pi_P)}}}\\
	 & \le c_1t \sqrt{\frac1{n_Q} + \frac1{\eps_P n_P - c't \sqrt{ \eps_Pn_P } }}\\
	 & \le c_2t \sqrt{\frac1{n_Q} + \frac1{\eps_P n_P }}\\
	 & \le c_3 \left(\frac1{\sqrt{n_Q}} + \frac1{\sqrt{n_P\eps_P}}\right)
	\end{aligned} \] for suitably chosen constants. This implies $ \hat \pi_Q $ converges to $ \pi_Q $ at a $ \frac1{\sqrt{n_Q}} + \frac1{\sqrt{n_P\eps_P}} $ rate.

    From lemma \ref{lemma:rate-g1} we see that $\hat g_1^{(n_1)}$ converges pointwise to   $g_1$ at a $(n_P\pi_P)^{-\beta/(2\beta + d)}$-rate and $\hat g_0^{(n_0)}$ converges pointwise to   $g_0$ at a $\big(n_P(1-\pi_P)\big)^{-\beta/(2\beta + d)}$-rate. Rest of the proof is same as in the proof of \ref{proof:upper-bound-labeled}.
\end{proof}

\section{Additional detail for proof of lower bounds}
\label{sec:more-lb-proofs}

\subsection{Proof of Theorem \ref{th:lower-bound}}

    The proof has two parts. The first part shows that the minimax rate is at least $(\eps_P n_P + n_Q)^{-\beta(1+\alpha)/(2\beta+d)}.$ This arises from the difficulties of estimating the non-parametric parts $g_0$ and $g_1.$ The second part shows that the minimax rate is at least $n_Q^{-(1+\alpha)/2}.$ This part is relatively easy to work with and mainly stems from the estimation of parametric part $\pi_Q.$ 
    
    \subsubsection{Difficulty of the non-parametric part} The crux of this part lies in construction of a family of distribution $\{\Pi_i\}_{i=1}^M.$ 
    
    Let $r = c_r (\eps_Pn_P+n_Q)^{-1/(2\beta + d)}, m = \lfloor c_m r ^{\alpha\beta - d}\rfloor,$ where $\alpha\ge 0 $ is the noise condition exponent and $\beta>0 $ is the H\"older smoothness exponent.

    Also, define  $c_r = (1/17) ,\  c_m = 8 \times 17^{\alpha\beta - d}$ and $c_w \in (0, 1)$ is a constant to be picked later. 
    The preceding constants satisfy \[ 8 \le m <\frac12 \left\lfloor \frac 1r \right\rfloor^d. \]  
    To see the first inequality, we observe that $r \le 1/17$ and recall $\alpha\beta \le 1 \le  d.$ 
    Thus \[ m = 8 \times (17r)^{\alpha\beta - d} \ge 8. \] 
    To see the second inequality, we observe that $r^{-1} \ge 16,$ which implies $r^{-1}\le 17\lfloor r^{-1}\rfloor/16.$ 
    Thus \[ m = 8 (17r)^{\alpha\beta} \left(\frac 1{17r}\right)^d \le 8\cdot  16^{-d} \lfloor r^{-1} \rfloor ^d \le\frac1 2 \lfloor r^{-1} \rfloor ^d. \] We also have $2mw = 2mc_w r^d\le 2c_m c_w<1 $ for a suitable $c_w.$
    
    \textit{Construction of $\{\Pi_i\}_{i=1}^M$.}
    Let $r_1 = 1/\lfloor 1/(c_wr)\rfloor$ if $ \lfloor 1/(c_wr)\rfloor $ is even, otherwise let $ r_1 = 1/\left(\lfloor 1/(c_wr)\rfloor+1\right). $
    Let us consider the grid of points \begin{equation}
    \cZ = \{ (1/2+i)r_1 : i = 0, 1, \dots , 1/r_1 -1 \}^d.
    \end{equation} 
    We see that $\cZ$ is a grid of equally spaced points of size $r_1^{-d}.$ For a $z \in \cZ$ we consider the hyper-cube \[ C(z) = \{x \in [0,1]^d: \|x-z\|_\infty\le r_1/2 \}. \] Note that, volume of each of these hyper-cubes is $r_1^d.$ Let $\cZ_1, \cZ_2 \subset \cZ $ be subsets of size $m$. Moreover, we let $\cZ_1 $ and $\cZ_2$ are disjoint. We define a bijection $u : \cZ_1 \to \cZ_2 $ which shall be used to construct the conditional densities. We define $\cZ_0 = \cZ \backslash (\cZ_1 \cup \cZ_2).$  Note that, $ \cZ $ has even number of points and $ |\cZ_1| = |\cZ_2| .$ Hence,  $ \cZ_0 $ has even number of points. We further divide $ \cZ_0 $ in two sets $ \cZ_3, \cZ_4 $ of equal sizes.  We shall define a set of distributions parametrized by $\sigma\in \{ -1, 1 \}^{\cZ_1}.$ 

    \noindent\textit{Conditional densities.}  For $a >0$ we define a function $v_a$ supported on on $\reals$ which will be used heavily for the construction of conditional densities.

    Define   
    \[u_{a}(x)=     
    \begin{cases}
    0 & \text{ for }x<0\\
    \frac{ \int_{0}^x e^{-\frac1{at(1-t)}}dt}{\int_{0}^1 e^{-\frac1{at(1-t)}}dt} & \text{ for }0 \le x \le 1\\
    1 & \text{ for } x>1
    \end{cases}\]
    and 
    \begin{equation}
    \label{eq:u-function}
    v_{a}(x)=     \begin{cases}    
    \big( 1-u_a(x) \big)^{1/\alpha}  & \quad \text{for }\beta < 1,\\
     \big( 1 -u_a(x) \big) & \quad \text{for }\beta \ge 1.
    \end{cases}
    \end{equation}

    According to lemma \ref{lemma:smooth-u-fn} we choose $ a $ such that $ v_a \equiv v $ is $ (\beta, C_\beta) $ H\"older smooth.  
    Therefore, the following functions are $(\beta, C_\beta)$-H\"older smooth: 
    \[ z \in \cZ, \hspace{0.3cm} \eta_z(x) =  \frac{\mu_\Delta r^\beta}3 v \left(\frac{2\|x-z\|_\infty }{r_1}\right) \]
    and     \[ z \in \cZ, \hspace{0.3cm} \xi_z(x) =   v \Bigg(\frac{2\|x-z\|_\infty }{br_1} - \frac2{r_1}\left(\frac1b-1\right)\Bigg). \] Here $ \mu_\Delta<1 $ is chosen later.  
    For a parameter $\sigma$ the  construction of  conditional densities are given below.

    \begin{gather}
    \left\{    
    \begin{aligned}
    g_1^{\sigma}(x) & = \begin{cases}
    1+  \sigma(z) \sqrt{\eps_P} \eta_z(x) & \hspace{0.3cm} x \in C(z),\ z \in \cZ_1,\\
    1-  \sigma(z) \sqrt{\eps_P} \eta_{f(z)}(x) & \hspace{0.3cm} x \in C(f(z)),\  z \in \cZ_1,\\
    1 + \xi_z(x) & \hspace{0.3cm} x \in C(z),\  z \in \cZ_3,\\
    1 - \xi_z(x) & \hspace{0.3cm} x \in C(z),\  z \in \cZ_4,
    \end{cases} \\
    g_0^{\sigma}(x)&  = \begin{cases}
    1-  \sigma(z)   \eta_z(x) & \hspace{0.3cm} x \in C(z),\ z \in \cZ_1,\\
    1+  \sigma(z)   \eta_{f(z)}(x) & \hspace{0.3cm} x \in C(f(z)),\  z \in \cZ_1,\\
    1 - \xi_z(x) & \hspace{0.3cm} x \in C(z),\  z \in \cZ_3,\\
    1 + \xi_z(x) & \hspace{0.3cm} x \in C(z),\  z \in \cZ_4.\end{cases}
    \end{aligned}
    \right.
    \label{eq:density-class-supervised-density}
    \end{gather}

    We also define $\pi_Q^{\sigma} = 1/2$ and $\pi_P^\sigma = 1-\eps_P.$ We then define the probabilities \[P_{\sigma}(X \in A, Y= y) = \int_A [\pi_P^{\sigma} g^{\sigma}_1(x)\indicator(y = 1) + (1-\pi_P^{\sigma}) g^{\sigma}_0(x) \indicator (y = 0) ] dx\] and
    \begin{equation}
    Q_{\sigma}( X \in A, Y = y) = \int_A [\pi_Q^
    {\sigma}g^{\sigma}_1(x)\indicator(y = 1) + (1-\pi^{\sigma}_Q) g^{\sigma}_0(x) \indicator (y = 0) ] dx.
    \label{eq:dist-Q-labeled-density}
    \end{equation}
    Given the source and target distributions we define the joint distribution of $\dataunlabeled$ as \begin{equation}
    \label{eq:dist-labeled-joint-density}
    \Pi_\sigma = P_\sigma^{\otimes n_P} \otimes Q_{\sigma, X} ^{\otimes n_Q}
    \end{equation}
    
    Here, $\eps_P \le \pi_P^\sigma \le 1-\eps_P.$ Also,  for any $x \in \Omega.$ Hence, $\mu_- \le q^{\sigma}_X(x) \le  \mu _+ $ for a suitable $ \mu_\Delta. $ Furthermore, $\Omega = [0,1]^d$ is a regular set. Hence, $q^\sigma _X$ satisfies strong density assumption.

    For such a construction we refer to lemma 6.4, where it is shown $Q_\sigma$ satisfies $\alpha$-margin condition with constant $C_\alpha.$

    Let $ \cF $ be the set of all classifier relevant to this classification problem. For $\sigma \in \{-1, 1\}^{\cZ_1}$ let  $f_\sigma$ be the Bayes classifier corresponding to the probability distribution $Q_\sigma$ defined as $ f_\sigma (x) = \indicator\{\eta_{ Q_\sigma }(x) \ge 1/2\}.$  For $ \sigma,\sigma'\in\{-1, 1\}^{\cZ_1} $ define $ \bar\rho (\sigma,\sigma') \coloneqq \cE_{\sigma} (f_{\sigma'})   $ and  $\rho(\sigma,\sigma') = \text{card} \{z\in\cZ_1 :\sigma(z) \neq \sigma'(z) \} $ as the Hamming distance. Then \begin{align*}
    \bar\rho (\sigma,\sigma') & = 2\Ex _{Q_{\sigma, X}} \left[\left | \eta_Q^\sigma(X) - \frac12 \right | \mathbf{1}\left( f_\sigma(X) \neq f_{\sigma'}(X)  \right)\right]\\
    & \ge c_1   r_1^d r^\beta  \rho(\sigma, \sigma')\\
    & \ge c_1 c_w^d r^{\beta+d}\rho(\sigma, \sigma').
    \end{align*}
    
    We recall Varshamov-Gilbert bound, which shall be used to construct the probability class.

    Let $\{ \sigma_0,\ldots , \sigma_M \} \subset \{-1,1\}^m  $ be the choice obtained from the  lemma \ref{lemma:varshamov-gilbert}.  Note that for such a choice $\rho(\sigma_i, \sigma_j) \ge m/8$ whenever $i\neq j.$
    
    Then  \begin{align*}
    \bar\rho (\sigma_i, \sigma_j) &\ge c_1c_w^d  r^{\beta+d}  \frac m8\\ 
    &\ge c_1c_w^d  r^{\beta+d}   r^{\alpha\beta -d}  \\
    & \ge c' r^{\beta (1+\alpha)}\\
    & = c' (\eps_P n_P)^{-\frac{\beta(1+\alpha)}{2\beta + d}}\\
    & \triangleq 2s
    \end{align*}

    Now we bound the Kulback-Leibler divergence between the joint distributions $\Pi_{\sigma_i}.$ Using lemma \ref{lemma:kl-distance-joint} we get 
    \begin{align*}
    KL(\Pi_{\sigma_i}|| \Pi_{\sigma_j}) & \le c_w^dK(d, \alpha, \beta)\rho(\sigma_i, \sigma_j) \\
    & \le  c_w^dK(d, \alpha, \beta)m \\
    & \le \frac 19 \log_2(M) \\
    \end{align*}
    for suitable $c_w < 1.$

    Finally we appeal to proposition 6.1 (and Markov's inequality) to obtain the minimax rate \begin{align*}
    \sup_{(P,Q)\in \Pi} \Ex \cE_Q (\hat f)& \ge \sup_{(P,Q)\in \Pi} s \Pr_\Pi \left ( \cE_Q (\hat f)\ge s \right)\\
    & \ge s \sup_{\sigma\in \{-1,1\}^{\cZ_1}} \Pi_\sigma \left(\cE_{Q^\sigma} (\hat f)\ge s \right) \\
    & \ge s \frac {3-2\sqrt{2}}{8}\\
    & \ge C(\eps_Pn_P)^{-\frac{\beta(1+\alpha)}{2\beta + d}}.
    \end{align*}

Proof of the parametric part will be exactly same as in the proof of theorem 4.1, where we get the lower bound
\[
    \sup_{(P, Q) \in \Pi} \Ex \big[ \cE_Q(\hat f) \big] \ge c n_Q^{-\frac{1+\alpha}{2}}.
\]

Finally, we combine the two bounds to get \[ 
\begin{aligned}
\sup_{(P, Q) \in \Pi} \Ex \big[ \cE_Q(\hat f) \big] & \ge c n_Q^{-\frac{1+\alpha}{2}} \vee c(\eps_Pn_P + n_Q)^{-\frac{\beta(1+\alpha)}{2\beta + d}}\\
& \ge c' \Bigg( (\eps_Pn_P + n_Q)^{-\frac{\beta}{2\beta + d}} + \frac{1}{\sqrt{n_Q}}  \Bigg)^{1+\alpha} .
\end{aligned}
\]

\subsection{Additional Lemmas}

\begin{lemma}
\label{lemma:inequality-1}
For any $0 \le x \le 1/3$ we have \[ \log\left(\frac{1+x}{1-x}\right) \le 3x . \] 
\end{lemma}

\begin{proof}
    Note that for $0 \le x\le 1/3$ we have $x - 3x^2 \ge 0.$ Hence, \begin{align*}
        1+x & \le 1 + 2x -3x^2 \\
         & \le (1-x)(1+3x)\\
         & \le (1-x) e^{3x}.\end{align*} Taking logarithm in both sides we have the result. 
\end{proof}

\begin{lemma}
\label{lemma:kl-conditional-density}
Let $|e|\le 1.$ For $\sigma \in \{-1, 1\}^{\cZ_1}$ let $p_\sigma$ be a probability density function \[ p_{\sigma}(x) = \begin{cases}
1+  \sigma(z)e \eta_z(x) & \hspace{0.3cm} x \in C(z),\ z \in \cZ_1,\\
1-  \sigma(z) e \eta_{f(z)}(x) & \hspace{0.3cm} x \in C(f(z)),\  z \in \cZ_1,\\
1 + \xi_z(x) & \hspace{0.3cm} x \in C(z),\  z \in \cZ_3,\\
1 - \xi_z(x) & \hspace{0.3cm} x \in C(z),\  z \in \cZ_4,
\end{cases} , \] as defined in \ref{eq:density-class-supervised-density} and Section 6, in main document, equation 6.3. For $\sigma, \sigma'\in \{-1, 1\}^{\cZ_1}$ let us define the Hamming distance as $\rho(\sigma, \sigma') = \sum_{z\in \cZ_1} \indicator_{\{\sigma(z) \neq \sigma'(z)\}}.$ Then for $\sigma, \sigma'\in \{-1, 1\}^{\cZ_1}$ \begin{equation}
        KL(p_{\sigma}||p_{\sigma'}) \le e^2 c_w^d K(d, \alpha, \beta) r^{2\beta+d}\rho(\sigma, \sigma'),
    \end{equation} for some constant $K(d, \alpha, \beta)$ only being dependent on $d, \alpha$ and $beta.$
\end{lemma}

\begin{proof}

\begin{align*}
 KL(p_{\sigma}||p_{\sigma'})  = & \int \log\left( \frac{p_\sigma(x)}{p_{\sigma'}(x)} \right) p_\sigma (x) dx\\
= & \sum_{z \in \cZ_1} \int_{C(z)} \log\left( \frac{1 + \sigma(z)e\eta_z(x)}{1 + \sigma'(z)e\eta_z(x)} \right)  (1 + \sigma(z)e\eta_z(x))dx \\
    +& \sum_{z \in \cZ_1} \int_{C(f(z))} \log\left( \frac{1 - \sigma(z)e\eta_{f(z)}(x)}{1 - \sigma'(z)e\eta_{f(z)}(x)} \right)  (1 - \sigma(z)e\eta_{f(z)}(x))dx\\
= & \sum_{z \in \cZ_1} \int_{C(z)} \log\left( \frac{1 + \sigma(z)e\eta_z(x)}{1 + \sigma'(z)e\eta_z(x)} \right)  (1 + \sigma(z)e\eta_z(x))dx \\
     & + \sum_{z \in \cZ_1} \int_{C(z)} \log\left( \frac{1 - \sigma(z)e\eta_{z}(x)}{1 - \sigma'(z)e\eta_{z}(x)} \right)  (1 - \sigma(z)e\eta_{z}(x))dx \\
= & \sum_{\sigma(z)\neq \sigma'(z)} \int_{C(z)} \log\left( \frac{1 + \sigma(z)e\eta_z(x)}{1 + \sigma'(z)e\eta_z(x)} \right)  (1 + \sigma(z)e\eta_z(x))dx \\
     & + \sum_{\sigma(z)\neq \sigma'(z)} \int_{C(z)} \log\left( \frac{1 - \sigma(z)e\eta_{z}(x)}{1 - \sigma'(z)e\eta_{z}(x)} \right)  (1 - \sigma(z)e\eta_{z}(x))dx 
\end{align*}
It's easy to see that above expression is invariant with respect to the sign of $b.$ So, we assume $b>0.$

\begin{align*}
 KL(p_{\sigma}||p_{\sigma'})  = & \sum_{\sigma(z)\neq \sigma'(z)} \int_{C(z)} \log\left( \frac{1 + \sigma(z)e\eta_z(x)}{1 - \sigma(z)e\eta_z(x)} \right)  (1 + \sigma(z)e\eta_z(x))dx \\
     & + \sum_{\sigma(z)\neq \sigma'(z)} \int_{C(z)} \log\left( \frac{1 - \sigma(z)e\eta_{z}(x)}{1 + \sigma(z)e\eta_{z}(x)} \right)  (1 - \sigma(z)e\eta_{z}(x))dx \\
    = & \sum_{\sigma(z)\neq \sigma'(z)} \int_{C(z)} \log\left( \frac{1 + \sigma(z)e\eta_z(x)}{1 - \sigma(z)e\eta_z(x)} \right)  2 \sigma(z)e\eta_z(x)dx \\
    = & \rho(\sigma, \sigma') \int_{C(z)} \log\left( \frac{1 + e\eta_z(x)}{1 - e\eta_z(x)} \right)  2 e\eta_z(x)dx\\
    \le  &\frac 23 \rho(\sigma, \sigma') \int_{C(z)} e^2\eta_z^2(x)dx, \quad \text{using \ref{lemma:inequality-1}}, \\
    = & K(d, \alpha, \beta) e^2 r^{2\beta}r_1^d \rho(\sigma, \sigma')   \\
    = &  K(d, \alpha, \beta) e^2 c_w^dr^{2\beta+d}\rho(\sigma, \sigma').
\end{align*}

\end{proof}

\begin{lemma}
\label{lemma:kl-decomposition}
Let $P, Q$ be two probability distributions defined on $ \Omega\times \{0, 1\}$ defined as \[P(X \in A, Y= y) = \int_A [\pi_P p_1(x)\indicator(y = 1) + (1-\pi_P) p_0(x) \indicator (y = 0) ] dx\] and \[Q( X \in A, Y = y) = \int_A [\pi_Q q_1(x)\indicator(y = 1) + (1-\pi_Q) q_0(x) \indicator (y = 0) ] dx\] for any $y \in \{0, 1\}$ and Borel subset $A$ of $\Omega,$ where $0 \le \pi_P, \pi_Q \le 1$ and $p_0, p_1, q_0, q_1$ are probability densities defined on $\Omega.$ Let $U\sim \text{Ber}(\pi_P)$ and $V\sim \text{Ber}(\pi_Q).$ Then  \[KL(P||Q) = KL(U|| V) + \pi_P KL(p_1||q_1) + (1-\pi_P) KL(p_0||q_0). \]
\end{lemma}

\begin{proof}
    Let $(X, Y) $ be a generic pair following the distributions $P$ and $Q.$ Then the lemma directly follows from the decomposition  \[KL(P||Q) = KL(P_Y||Q_Y) + E_{P_Y}\big[KL(P_{X|Y}||Q_{X|Y})\big].\]
\end{proof}

\begin{lemma}
\label{lemma:margin-condition}
For $\sigma \in \{-1, 1\}^{\cZ_1}$ let $Q_\sigma$ be the distribution defined in \ref{eq:dist-Q-labeled-density}.  Then for any $\sigma,$  $Q_\sigma$ satisfies $\alpha$-margin condition with constant $C_\alpha.$
\end{lemma}

\begin{proof}

 We recall the conditional densities \ref{eq:density-class-supervised-density}   \[g_1^{\sigma}(x) = \begin{cases}
1+  \sigma(z) \sqrt{\eps_P} \eta_z(x) & \hspace{0.3cm} x \in C(z),\ z \in \cZ_1,\\
1-  \sigma(z) \sqrt{\eps_P} \eta_{f(z)}(x) & \hspace{0.3cm} x \in C(f(z)),\  z \in \cZ_1,\\
1 + \xi_z(x) & \hspace{0.3cm} x \in C(z),\  z \in \cZ_3,\\
1 - \xi_z(x) & \hspace{0.3cm} x \in C(z),\  z \in \cZ_4,
\end{cases} \] and 
\[
g_0^{\sigma}(x) = \begin{cases}
1-  \sigma(z) \sqrt{\eps_P} \eta_z(x) & \hspace{0.3cm} x \in C(z),\ z \in \cZ_1,\\
1+  \sigma(z) \sqrt{\eps_P} \eta_{f(z)}(x) & \hspace{0.3cm} x \in C(f(z)),\  z \in \cZ_1,\\
1 - \xi_z(x) & \hspace{0.3cm} x \in C(z),\  z \in \cZ_3,\\
1 + \xi_z(x) & \hspace{0.3cm} x \in C(z),\  z \in \cZ_4,
\end{cases}
\]

From \[\eta_Q^\sigma(x) - \frac12 = \frac{\pi_Q g_1^\sigma (x) - (1-\pi_Q^\sigma)g_0^\sigma(x)}{2\pi_Q g_1^\sigma (x) + 2(1-\pi_Q^\sigma)g_0^\sigma(x)} = \frac{g_1^\sigma(x) - g_0^\sigma(x)}{2(g_1^\sigma(x)+g_0^\sigma(x))}\] we get  
\[
\eta_Q^\sigma(x) - 1/2 = \begin{cases} \frac{ (1+\sqrt{\eps_P})\sigma(z)\eta_z(x) }{2 - (1-\sqrt{\eps_P})\sigma(z)\eta_z(x)} & \quad x \in C(z),\ z \in \cZ_1,\\
\frac{-(1+\sqrt{\eps_P} )\sigma(z)\eta_{f(z)}(x) }{2 + (1-\sqrt{\eps_P} )\sigma(z)\eta_{f(z)}(x)} & \quad x \in C(f(z)),\  z \in \cZ_1,\\
\frac12 \xi_z(x) & \quad x \in C(z),\  z \in \cZ_3,\\
-\frac12 \xi_z(x) & \quad x \in C(z),\  z \in \cZ_4.
\end{cases}
\]

Note that \[ \big| \eta_Q^\sigma(x) - 1/2 \big| \ge  \begin{cases} \frac14 \eta_z(x)  & \quad x \in C(z),\ z \in \cZ_1\cup \cZ_2,\\
\frac14 \xi_z(x) & \quad x \in C(z),\  z \in \cZ_0.
\end{cases} \]
Rest of the proof is similar as in the proof of lemma 6.4.
\end{proof}

\begin{lemma}
\label{lemma:kl-distance-joint}
Let $\Pi_\sigma$ be the probability distribution as defined in equation \ref{eq:dist-labeled-joint-density}. For $\sigma, \sigma' \in \{-1, 1\}^{\cZ_1}$ we have \[  KL(\Pi_\sigma|| \Pi_{\sigma'}) \le  2c_w^dK(d, \alpha, \beta)c_r^{2\beta + d}\rho(\sigma, \sigma'). \] 
\end{lemma}

\begin{proof}
	From lemmas \ref{lemma:kl-decomposition} and \ref{lemma:kl-conditional-density} we get 
	\begin{align*}
	    KL(Q_{\sigma}||Q_{\sigma'}) & = \frac 12 KL(g_1^{\sigma}||g_1^{\sigma'}) + \frac 12 KL(g_0^{\sigma}||g_0^{\sigma'})\\
	    & \le \frac 12 (\eps_P + 1) c_w^d K(d, \alpha, \beta) r^{2\beta+d}\rho(\sigma, \sigma') \\
	    & \le  c_w^d K(d, \alpha, \beta) r^{2\beta+d} \rho(\sigma, \sigma')
	\end{align*} and 
	\begin{align*}
	    KL(P_{\sigma}||P_{\sigma'}) & = (1-\eps_P) KL(g_1^{\sigma}||g_1^{\sigma'}) + \eps_P KL(g_0^{\sigma}||g_0^{\sigma'})\\
	    & \le  (\eps_P(1-\eps_P)+\eps_P) c_w^dK(d, \alpha, \beta) r^{2\beta+d}\rho(\sigma, \sigma') \\
	    & \le 2\eps_P c_w^dK(d, \alpha, \beta) r^{2\beta+d}\rho(\sigma, \sigma').
	\end{align*}
	Hence \begin{align*}
	    KL(\Pi_\sigma||\Pi_{\sigma'}) & = n_P KL(P_{\sigma}||P_{\sigma'}) + n_Q KL(Q_{\sigma}||Q_{\sigma'}) \\
	    & \le 2(\eps_Pn_P + n_Q) c_w^dK(d, \alpha, \beta) r^{2\beta+d}\rho(\sigma, \sigma') \\
	    & \le 2c_w^d K(d, \alpha, \beta)(\eps_Pn_P + n_Q)  (\eps_Pn_P + n_Q) ^{-\frac{2\beta + d}{2\beta + d}}\rho(\sigma, \sigma')\\
	   & \le  2c_w^dK(d, \alpha, \beta)c_r^{2\beta + d} (\eps_Pn_P + n_Q) (\eps_Pn_P + n_Q) ^{-\frac{2\beta + d}{2\beta + d}}\rho(\sigma, \sigma') \\
	   & \le 2c_w^dK(d, \alpha, \beta)c_r^{2\beta + d}\rho(\sigma, \sigma').
	\end{align*} 
\end{proof}

\begin{lemma}
	\label{lemma:smooth-u-fn}
	For any pair $ (\beta, C_\beta) $ of positive numbers there exists an $ M>0 $ such that the function $v_a $ (defined in equation \eqref{eq:u-function}) is $ \beta $ smooth with constant $ C_\beta $.  
\end{lemma}

\begin{proof}
	If $ \beta < 1 $ then \[ \begin{aligned}
	& \left|(1-u_a)^{1/\alpha}(x) - \big((1-u_{a})^{1/\alpha}\big)(y)\right|\\
	& \le \max_{y\ge 0} \frac{d}{dy}\big((1-u_a)^{1/\alpha}\big)(y) |x - x_0|\\
	& \le \frac{L}{a} |x - x_0|^{\beta}.
	\end{aligned} \]
	If $ \beta \ge 1 $
	 using Taylor's approximation theorem we get 
	\[ \begin{aligned}
	& \left|(1-u_a)(x) - (1-u_{a})_{x_0}^{(\beta)}\right|\\
	 & \le \max_{y\ge 0} \frac{(1-u_a)^{(\lfloor\beta\rfloor + 1)}(y)}{(\lfloor\beta\rfloor + 1)!} |x - x_0|^{(\lfloor\beta\rfloor + 1)}\\
	& \le \frac{L}{a} |x - x_0|^{\beta}.
	\end{aligned} \] for some $ L>0 ,$ if $ |x - x_0|\le 1. $ Hence, for a suitable $ a ,$  $ v_a $  is $ (\beta, C_\beta, 1) $ smooth. 
\end{proof}

\end{document}